\documentclass[10pt, a4paper, reqno, oneside]{amsart}

\usepackage{amsmath, amsopn, amsfonts, amsthm, amssymb, amscd, enumerate, multicol, scalefnt, mathtools}

\makeatletter
\@namedef{subjclassname@2020}{\textup{2020} Mathematics Subject Classification}
\makeatother

\usepackage{sansmath}
\usepackage{etex}
\usepackage{relsize}
\usepackage{hyperref}
\usepackage{t1enc}
\usepackage{relsize}
\usepackage[latin1]{inputenc}
\usepackage{graphicx}
\usepackage[all]{xy}
\usepackage{color}
\usepackage{slashed}
\usepackage[mathscr]{euscript}
\usepackage{enumitem}
\setlist[itemize]{noitemsep, topsep=1pt, leftmargin=20pt}
\newcommand\bcdot{\ensuremath{
  \mathchoice
   {\mskip\thinmuskip\lower0.2ex\hbox{\scalebox{1.6}{$\cdot$}}\mskip\thinmuskip}}
   {\mskip\thinmuskip\lower0.2ex\hbox{\scalebox{1.6}{$\cdot$}}\mskip\thinmuskip}
   {\lower0.3ex\hbox{\scalebox{1.2}{$\cdot$}}}
   {\lower0.3ex\hbox{\scalebox{1.2}{$\cdot$}}}
}
\theoremstyle{plain}
\newtheorem{theo}{Theorem}[section]

\newtheorem{prop}[theo]{Proposition}

\theoremstyle{definition}
\newtheorem{rem}[theo]{Remark}

\newtheorem{definition}[theo]{Definition}

\theoremstyle{plain}
\newtheorem{lemma}[theo]{Lemma}
\newtheorem{theorem}[theo]{Theorem}

\theoremstyle{definition}

\theoremstyle{plain}

\newcounter{ssig}
\newcounter{ttig}
\title[]{Sobolev regularity for nonlinear Poisson equations with Neumann boundary conditions on Riemannian manifolds}

\author{Alessandro Goffi}
\address[Alessandro Goffi]{Dipartimento di Matematica ``Tullio Levi-Civita'', Universit\`a di Padova, Via Trieste 63, 35121 Padova, Italy}
\email{alessandro.goffi@unipd.it}

\author{Francesco Pediconi}
\address[Francesco Pediconi]{Department of Mathematics, Aarhus University, Ny Munkegade 118, 8000 Aarhus C, Denmark}
\email{francesco.pediconi@math.au.dk}


\thanks{The first-named author is member of GNAMPA of INdAM and has been partially supported by the GNAMPA project ``Mean Field Games: modelli e sviluppi''. The second-named author is member of GNSAGA of INdAM and has been supported by the project PRIN 2017 ``Real and Complex Manifolds: Topology, Geometry and holomorphic dynamics'' (code 2017JZ2SW5).}

\begin{document}
\begin{abstract}
In this paper, we study the Sobolev regularity of solutions to nonlinear second order elliptic equations with super-linear first-order terms on Riemannian manifolds, complemented with Neumann boundary conditions, when the source term of the equation belongs to a Lebesgue space, under various integrability regimes. Our method is based on an integral refinement of the Bochner's identity, and leads to ``semilinear Calder\'on-Zygmund'' type results. Applications to the problem of smoothness of solutions to Mean Field Games systems with Neumann boundary conditions posed on convex domains of the Euclidean space will also be discussed.
\end{abstract}

\maketitle


\section{Introduction}
\setcounter{equation} 0

The purpose of this paper is to establish Sobolev a priori estimates for solutions to the Neumann problem for semi-linear elliptic equations on a Riemannian manifold $(M,g)$ whose prototype is
\begin{equation} \label{eqintro}
\begin{cases}
-\Delta u(x)+H(x,{\rm d}u(x))=f(x) & \text{ in } \Omega \subset M \\
\partial_\nu u(x) =0 & \text{ on } \partial\Omega
\end{cases}
\end{equation}
where $\Delta$ is the Laplace-Beltrami operator of $(M,g)$, $H: T^*M \to \mathbb{R}$ is a nonlinear function which has superlinear growth $\gamma>1$ in the second entry, $\Omega \subset M$ is a domain satisfying suitable geometric conditions (see Remark \ref{classO}), $\partial_\nu$ is the normal derivative and the source term of the equation $f$ is unbounded and belongs to some Lebesgue space $L^q(\Omega,\mathbb{R})$. Such PDEs with superlinear first-order terms appear naturally in several fields, such as Stochastic Control Theory \cite{BF} (where \eqref{eqintro} builds upon a stochastic differential equation with controlled drift and reflection at the boundary), Differential Geometry \cite{EellsSampson} and in the recent literature of Mean Field Games \cite{LL07}.

In particular, we will focus on two type of gradient estimates for solutions of such equations. First, we ask which integrability on the source term $f$ is necessary to achieve $L^p$-gradient bounds. Second, we study an optimal $L^q$-regularity estimate, and ask whether for some integrability range $q>\bar q>1$, a control of $f \in L^{q}(\Omega,\mathbb{R})$ implies bounds on the same Lebesgue space for the terms on the left-hand side of \eqref{eqintro}, meaning that there is no loss of regularity in the equation.

Since we suppose that solutions to our problem are smooth (or at least strong, see Remark \ref{remClass-Str}), our results are closer to the domain of a priori estimates rather than to regularity ones, as it happens in Calder\'on-Zygmund regularity theory (see  \cite[Chapter 9]{GT}).\smallskip

Our starting point is the so-called Bochner Identity, which states if $(M,g)$ is a $d$-dimensional Riemannian manifold and $u: M \to \mathbb{R}$ is of class $\mathcal{C}^3$, one has
$$
\Delta\big(\tfrac12|\nabla u|^2\big) = g(\nabla(\Delta u),\nabla u) + |D^2u|^2 + \mathrm{Ric}(\nabla u,\nabla u) \,\, ,
$$
where $D$ denotes the Levi-Civita connection, $\nabla$ the Riemannian gradient and $\mathrm{Ric}$ the Ricci curvature of $(M,g)$. Then, following the notation from Gamma-calculus \cite{BGL}, if we consider the symmetric bilinear forms on $\mathcal{C}^{\infty}(M,\mathbb{R})$ defined by
$$
\Gamma(u,v) \coloneqq g(\nabla u, \nabla v) \,\, , \quad
\Gamma_2(u,v) \coloneqq \tfrac12 \big(\Delta(\Gamma(u,v))-\Gamma(u,\Delta v)-\Gamma(v,\Delta u)\big) \,\, ,
$$
together with the associated quadratic forms $\Gamma(u) \coloneqq |\nabla u|^2$ and $\Gamma_2(u) \coloneqq \Gamma_2(u,u)$, it is immediate to realize that the Bochner Identity implies that
$$
\Gamma_2(u) = |D^2u|^2 +\mathrm{Ric}(\nabla u,\nabla u) \,\, .
$$
Therefore, the Cauchy-Schwarz Inequality and a uniform lower bound on the Ricci curvature $\mathrm{Ric} \geq -\kappa g$ give the so-called {\it curvature-dimension inequality} (see \cite{BGL})
\begin{equation}\label{cdineq}
\Gamma_2(u) \geq \tfrac1d (\Delta u)^2 -\kappa\,\Gamma(u) \,\, .
\end{equation}
This simple and deep algebraic inequality plays a pivotal role in deriving quantitative results such as gradient estimates, which in the PDE literature are part of the so-called {\it Bernstein gradient estimates}, see e.g. \cite{ll, Lions85, VeronBook, VeronJFA, CGell}. As it is well-known, these bounds are the lynchpin to derive many other qualitative and quantitative properties of solutions to nonlinear PDEs, such as Harnack inequalities and Liouville theorems \cite{SerrinPeletier,ColdingMinicozzi}, and even compactness theorems for linear and nonlinear PDEs. Actually, even beyond the application to PDEs, the curvature-dimension inequality, combined with the properties of the heat semigroup, is the cornerstone to connect the geometry of the manifold with some of its global properties, see \cite{GarofaloFrac,BaudoinGarofalo} and the references therein. \smallskip

In this paper, starting from the above observations, due to the presence of data belonging to Lebesgue spaces, the application of the pointwise Bernstein argument via the maximum principle (see e.g. \cite{PorLeo}) is ruled out, and one needs to develop an integral method. To this aim, with a slight abuse of terminology, we develop a kind of nonlinear integral curvature-dimension inequalities or, alternatively, integral Bernstein arguments, to study some regularity aspects in the scale of Sobolev spaces for nonlinear equations of the form \eqref{eqintro}. To do this, we borrow several tools from \cite{Lions85,ll,BardiPerthame} and \cite{CGell}, where similar regularity properties have been analyzed through the so-called integral Bernstein method and its refinements.

Our first main estimate, that is the content of Theorem \ref{main1}, gives the following. When $f \in L^q(\Omega,\mathbb{R})$ for some $q>d = \dim M$ and
$$
H(x,{\rm d}u(x))=\tfrac1\gamma|\nabla u|^\gamma+g(B,\nabla u) \,\, ,
$$
with $\gamma>1$ and $B \in L^s(\Omega,TM|_{\Omega})$ for some $s>d$, we prove that $\nabla u \in L^r(\Omega,TM|_{\Omega})$ for any $1 \leq r < \infty$. This step is achieved by first testing \eqref{cdineq} by suitable powers of $\Gamma(u)=|\nabla u|^2$ and then plugging the equation \eqref{eqintro} in \eqref{cdineq} to gain additional coercivity, that is
$$
|D^2u|^2 \geq \tfrac1d(\Delta u)^2 = \tfrac1d(H-f)^2\sim c_1|\nabla u|^{2\gamma}-c_2f^2\ .
$$
This step is crucial in many results in the PDE literature, both from a quantitative perspective, see \cite{Lions85,ll,CirantJMPA,PorLeo} among others, as well as for qualitative aspects like Liouville theorems, cf \cite{SerrinPeletier,Lions85,VeronBook}. 
The boundary terms are handled assuming a condition on the second fundamental form of the boundary $\partial \Omega$, which in the Euclidean case reduces to the convexity, that leads to the ``good sign'' for the exterior derivative term $\partial_\nu|\nabla u|^2$, see \cite{Lio80,CirantJMPA,PorrCCM} for similar assumptions and contexts. Then, a combination of H\"older, Young and Sobolev inequalities allows to get the conclusion through classical absorption schemes typical of regularity analyses for linear and nonlinear PDEs, see Section \ref{sect_proof1} for the details of the estimates. The first result dates back to \cite{Lions85} when $(M,g) = (\mathbb{R}^d,\langle\,,\rangle)$ equipped with Dirichlet boundary conditions, see also \cite{ll} for local estimates when $f\in L^\infty$, both of them when $B=0$. The presence of a drift term has been addressed in \cite{BardiPerthame} for elliptic equations with leading operator in non-divergence form with Sobolev diffusion matrix and first-order terms with natural growth. Finally, when $B=0$ and $(M,g) = (\mathbb{R}^d,\langle\,,\rangle)$ with $\Omega \subset \mathbb{R}^d$ convex, or $B=0$ and $(M,g) = (\mathbb{T}^d = \mathbb{R}^d/\mathbb{Z}^d, \langle\,,\rangle)$ with $\Omega = \mathbb{T}^d$, the results can be found respectively in \cite{CirantJMPA,PV}. \smallskip

Our second main estimate, that is the content of Theorem \ref{main2}, is much more delicate and exploit a further generalized version of \eqref{cdineq}, following the lines of \cite{CGell}. In particular, we prove that for a suitable $\mathcal{C}^2$-function $h=h(t)$ of one real variable, with $h'(t) \geq 0$ and $h''(t) \leq 0$ for any $t \geq 0$, we have
\begin{equation*}
\Delta \big(h\big(\tfrac12|\nabla u|^2\big)\big) = h'\big(\tfrac12|\nabla u|^2\big)\big( |D^2u|^2+ \Gamma(u,\Delta u) + \mathrm{Ric}(\nabla u,\nabla u)\big)+h''\big(\tfrac12|\nabla u|^2\big)|D^2u(\nabla u)|^2 \,\, , 
\end{equation*}
which leads, through the Cauchy-Schwarz inequality, to a weighted version of \eqref{cdineq}.
As a further additional step in the program, the integral version of the method requires to work on suitable restricted domains (i.e. on super-level sets of the gradient). Differently from Theorem \ref{main1}, the classical absorption scheme leads to an ``unbalanced'' inequality of the form
\[
y_k^{\frac{d-2}d} \leq y_k+\zeta(\mathrm{vol}(\{ |\nabla u| \geq k\})) \,\, , \quad y_k \sim \int_{\{|\nabla u|\geq k\}}(|\nabla u|-k)^{\gamma q}\,\omega_{\mathrm{vol}} \,\, ,
\]
for all $k \geq 1$, where $\zeta=\zeta(t)$ is a continuous function such that $\zeta(t) \to 0^+$ as $t \to 0^+$. Since ${\frac{d-2}d} <1$, one cannot apply the Young inequality to obtain the desired bound $|\nabla u|^\gamma \in L^q(\Omega,\mathbb{R})$. Nonetheless, one can get an estimate through a continuity argument introduced in \cite{CGell}, see also \cite{GMP14} for similar bounds in energy spaces of suitable powers of $|u|$. This allows to reach an estimate below the threshold $q=d$ and prove that
$$
f \in L^q(\Omega,\mathbb{R}) \,\, \text{ for some } q > \max\big\{d\tfrac{\gamma-1}{\gamma},2\big\} \quad \Longrightarrow \quad |\nabla u|^\gamma \in L^q(\Omega,\mathbb{R}) 
$$
and get even higher regularity for the second derivatives from the equation itself through Calder\'on-Zygmund theory, cf \cite{GT,PigolaSurvey,GunPig}. We remark that such an estimate fits within the so-called {\it maximal $L^q$-regularity estimates}, and has been proposed in a series of seminars by P.-L. Lions and recently addressed for problems posed on the flat torus, typical of ergodic control theory, in \cite{CGell} under the same integrability conditions. Here, as suggested by P.-L. Lions himself in his conferences, we continue the analysis initiated in \cite{CGell} and drop the periodicity condition, working with problems posed on Riemannian manifolds equipped with Neumann boundary conditions, providing a further technical advance in the field.

We emphasize that the proofs of Theorems \ref{main1} and \ref{main2} adapt, with minor modifications, to PDEs posed on compact Riemannian manifolds without boundary, for which the results to our knowledge are new (see Remark \ref{senzabdr}). A different approach to tackle interior bounds to such problems has been recently implemented in \cite{CV} through a blow-up method in the superquadratic regime. Let us finally remark that our methods do not allow to treat the parabolic version of \eqref{eqintro}. In this direction, rather different methods based on duality techniques recently appeared in \cite{CG2,CGpar}. \smallskip

As a byproduct of Theorem \ref{main2}, we show that such a level of regularity allows to improve the range of the smoothness of solutions to the following system of PDEs arising in the theory of Mean Field Games introduced by J.-M. Lasry and P.-L. Lions \cite{LL07}, posed on a convex domain $\Omega \subset \mathbb{R}^d$ of the Euclidean space, which describes the configuration at equilibrium of differential games with infinitely many indistinguishable rational agents:
\begin{equation*} \label{mfgintro}
\begin{cases}
-\Delta u(x) +H(x,{\rm d}u(x))+\lambda = m^\alpha & x \in \Omega \,\, , \\
-\Delta m(x) -\mathrm{div}\big(\partial_pH(x,{\rm d}u(x))m(x)\big)=0 & x \in \Omega \,\, , \\
\partial_\nu u (x) =0 \,\, , \quad \partial_\nu m(x) +\big(\partial_pH(x,{\rm d}u(x)) \cdot \nu\big)\,m(x) =0 & x \in \partial \Omega \,\, , \\
\int_\Omega m(x)\,{\rm d}x=1\ , \quad m(x)>0 & x \in \overline{\Omega} \,\, .
\end{cases}
\end{equation*}
In the above system, $H$ is the so-called Hamiltonian (while $\partial_pH(x,p)$ denotes the partial derivative of $H$ with respect to the second entry) and the variables $m: \Omega \to (0,+\infty)$, $u: \Omega \to \mathbb{R}$ and $\lambda \in \mathbb{R}$ describe respectively the equilibrium configuration of the agents, the equilibrium cost of a prototype player and the ergodic constant.
In our analysis, it is composed by a Hamilton-Jacobi equation with a Hamiltonian term having superlinear growth in the gradient entry, i.e. $H(x,{\rm d}u(x)) \sim |\nabla u|^{\gamma}$ for some $\gamma>1$, and a stationary Fokker-Planck equation driven by the (optimal) vector field $\partial_pH(x,{\rm d}u(x))\sim |\nabla u|^{\gamma-1} $, which heuristically describes the optimal strategies of the average player. In particular, we show in Theorem \ref{mainappl} that for such a model problem, there exists a classical solution $(u,\lambda,m)$ provided the exponent $\alpha$ satisfies
$$
\alpha < \tfrac{\gamma'}{d-2-\gamma'} \,\, , \quad \text{ where } \gamma' = \tfrac{\gamma}{\gamma-1} \,\, .
$$
This improves upon the known results in the literature for problems with Neumann boundary conditions, cf \cite{CirantJMPA,Alpar}, and also the better range found in the periodic setting \cite{CCPDE}, see Remark \ref{rem:exp}, and, finally, is flexible enough to cover problems with multi-populations.

\medskip
\noindent \textit{Plan of the paper}. In Section \ref{sect_2}, we collect some preliminaries of Differential Geometry and Geometric Analysis. Section \ref{sect_3} contains the statements and the proofs of our main results, that are Theorem \ref{main1} and Theorem \ref{main2}. Finally, in Section \ref{sect_4}, we provide the above mentioned application of our estimates, that is Theorem \ref{mainappl}.

\medskip
\noindent \textit{Acknowledgments}. This work has been written while the second-named author was a postdoctoral fellow at the Dipartimento di Matematica e Informatica ``Ulisse Dini'', Universit\`a di Firenze. He is grateful to the
department for the hospitality.

\medskip
\section{Preliminaries of Geometric Analysis} \label{sect_2}
\setcounter{equation} 0

\subsection{Riemannian geometry} \hfill \par

Let $(M^d,g)$ be a smooth Riemannian manifold of dimension $d \geq 3$. From now on, we will always assume that $(M,g)$ is connected, orientable and not necessarily complete.

We denote by $\mathrm{d}$ the exterior derivative, by $D$ the Levi-Civita covariant derivative and by $\omega_{\mathrm{vol}}$ the Rieman\-nian volume form. Furthermore, we denote by $\mathrm{Rm}(X,Y) \coloneqq D_{[X,Y]}-[D_X,D_Y]$ the Riemannian curvature tensor and by $\mathrm{Ric}(X,Y) \coloneqq \mathrm{Tr}\big(Z \mapsto \mathrm{Rm}(X,Z)Y\big)$ the Ricci tensor. Notice that the Riemannian volume form induces a measure ${\rm vol}$ on $M$ by setting $${\rm vol}(E) \coloneqq \textstyle\int_E \omega_{\mathrm{vol}} \quad \text{ for any Borel subset $E \subset M$ } \,\, .$$ We also denote by $|\cdot|$ the norm induced by $g$ on the tensor bundle over $M$. For example, given a field of endomorphisms $A \in \mathcal{C}^{0}(M,{\rm End}(TM))$, we have $$|A|: M \to \mathbb{R} \,\, , \quad |A|_x \coloneqq \Big(\sum_{i,j=1}^d g_x(A_x(e_i),e_j)^2\Big)^{\frac12} \,\, ,$$ where $(e_i)$ is any orthonormal frame for the tangent space $T_xM$. \smallskip

Given a sufficiently regular function $u: M \to \mathbb{R}$, we denote by $\nabla u$ the {\it gradient of $u$} defined by $$g(\nabla u,X) \coloneqq \mathrm{d}u (X) \quad \text{ for any } X \in \mathcal{C}^{\infty}(M,TM)$$ and by $D^2 u$ the {\it Hessian of $u$} defined by $D^2 u(X) \coloneqq D_X(\nabla u)$. We also denote by
$$
\Delta u \coloneqq \mathrm{Tr}\big(Z \mapsto D^2u(Z)\big) = {\rm div}(\nabla u)
$$
the {\it Laplacian of $u$}. Here, $\mathrm{div}(X) \coloneqq \mathrm{Tr}\big(Z \mapsto D_ZX\big)$ denotes the {\it divergence} of the vector field $X$. We also recall the following well-known

\begin{prop}[Bochner Formula]
Let $u \in \mathcal{C}^3(M,\mathbb{R})$ be a function and set $w \coloneqq \tfrac12|\nabla u |^2$. Then, it holds that
\begin{equation} \label{bw1}
\Delta w = g(\nabla(\Delta u),\nabla u) + |D^2u|^2 + \mathrm{Ric}(\nabla u,\nabla u) \,\, .
\end{equation}
\end{prop}

\begin{proof}
By definition we get $\mathrm{d} w(X) = g(D_X(\nabla u), \nabla u)$ for any smooth vector field $X$, hence
\begin{equation} \label{gradw}
\nabla w = D^2u(\nabla u) \,\, .
\end{equation}
Therefore, if we fix a point $x \in M$ and a local frame $(e_i)$ in a neighborhood of $x$ such that $(e_i)_x$ is an orthonormal basis for $T_xM$ and $(D e_i)_x=0$, from \eqref{gradw} we get
\begin{align*}
\Delta w &= \sum_{i=1}^n g(D_{e_i}D_{\nabla u} \nabla u , e_i) \\
&= \sum_{i=1}^n g(D_{\nabla u}D_{e_i} \nabla u , e_i) +\sum_{i=1}^ng(D_{[e_i,\nabla u]} \nabla u , e_i) -\sum_{i=1}^ng(\mathrm{Rm}(e_i,\nabla u)\nabla u,e_i) \\
&= \sum_{i=1}^n \mathcal{L}_{\nabla u}(g(D_{e_i} \nabla u, e_i)) +\sum_{i,j=1}^n g(D_{e_j} \nabla u, e_i)^2 +\sum_{i=1}^ng(\mathrm{Rm}(\nabla u,e_i)\nabla u,e_i) \\
&= g(\nabla(\Delta u),\nabla u) + |D^2u|^2 + \mathrm{Ric}(\nabla u,\nabla u)
\end{align*}
where $\mathcal{L}$ denotes the Lie derivative. This completes the proof.
\end{proof}

Notice that the Bochner Formula implies the following

\begin{lemma}\label{lemboc}
Let $u \in \mathcal{C}^3(M,\mathbb{R})$, $w \coloneqq \tfrac12|\nabla u |^2$, $h: [0,+\infty) \to \mathbb{R}$ of class $\mathcal{C}^2$ and set
$$z \coloneqq h(w) \,\, , \quad z^{(1)} \coloneqq h'(w) \,\, , \quad z^{(2)} \coloneqq h''(w) \,\, .$$
Then
\begin{equation} \label{bw2}
\Delta z = z^{(1)}\big(g(\nabla(\Delta u),\nabla u) + |D^2u|^2 + \mathrm{Ric}(\nabla u,\nabla u)\big)+z^{(2)}|D^2u(\nabla u)|^2 \,\, . 
\end{equation}
\end{lemma}

\begin{proof} Observe that
$$
\Delta z = \Delta \big(h \circ w \big) = (h' {\circ} w)\, \Delta w + (h'' {\circ} w)\, |\nabla w|^2 \,\, .
$$
Therefore, by \eqref{bw1} and \eqref{gradw} we obtain \eqref{bw2}. \end{proof}

Let now $\Omega \subset M$ be a {\it domain}, i.e. a non-empty, connected, open subset, and assume that the boundary $\partial \Omega$ is an orientable, embedded hypersurface of class $\mathcal{C}^{\infty}$. This implies that there exists a non-vanishing, unitary, normal, outward-pointing vector field $\nu \in \mathcal{C}^{\infty}(\partial \Omega,TM|_{\partial \Omega})$ on the whole $\partial \Omega$. Then, from the Stokes' Theorem we get the following

\begin{prop}[Integration by parts Formula]
Let us consider a function $u \in \mathcal{C}^1(\Omega,\mathbb{R}) \cap \mathcal{C}^0(\overline{\Omega},\mathbb{R})$ and a vector field $X \in \mathcal{C}^1(\Omega,TM|_{\Omega}) \cap \mathcal{C}^0(\overline{\Omega},TM|_{\overline{\Omega}})$. Then
\begin{equation} \label{intpart}
\int_{\Omega} g(\nabla u,X) \, \omega_{\mathrm{vol}} + \int_{\Omega} u \, \mathrm{div}(X) \, \omega_{\mathrm{vol}} = \int_{\partial \Omega} u\,g(X, \nu) \, \imath_{\partial\Omega}{}^*(\nu \lrcorner \omega_{\mathrm{vol}}) \,\, ,
\end{equation}
where $\imath_{\partial\Omega}$ denotes the canonical embedding of $\partial\Omega$ into $M$ and $\imath_{\partial\Omega}{}^*(\nu \!\lrcorner \omega_{\mathrm{vol}}) = \omega_{\mathrm{vol}}\big(\nu, \mathrm{d}\,\imath_{\partial\Omega}({\cdot}),\dots,\mathrm{d}\,\imath_{\partial\Omega}({\cdot})\big)$ is the induced Riemannian volume form on the boundary $\partial\Omega$. 
\end{prop}

\begin{proof}
We first observe that
$$\mathcal{L}_X(u\,\omega_{\mathrm{vol}}) = (\mathcal{L}_Xu)\,\omega_{\mathrm{vol}} + u\,(\mathcal{L}_X\omega_{\mathrm{vol}}) = g(\nabla u, X)\,\omega_{\mathrm{vol}} +u \, \mathrm{div}(X) \, \omega_{\mathrm{vol}} \,\, .$$
Moreover, by the Cartan Magic Formula
$$\mathcal{L}_X(u\,\omega_{\mathrm{vol}}) = \mathrm{d}(u\, X\lrcorner \omega_{\mathrm{vol}}) + X\lrcorner \mathrm{d}(u\,\omega_{\mathrm{vol}}) = \mathrm{d}(u\, X\lrcorner \omega_{\mathrm{vol}})$$
and so, by the Stokes Theorem
\begin{align*}
\int_{\Omega} g(\nabla u,X) &\, \omega_{\mathrm{vol}} + \int_{\Omega} u \, \mathrm{div}(X) \, \omega_{\mathrm{vol}} = \\
&= \int_{\Omega} \mathrm{d}(u\, X\lrcorner \omega_{\mathrm{vol}}) = \int_{\partial \Omega} \imath_{\partial\Omega}{}^*(u\, X\lrcorner \omega_{\mathrm{vol}}) = \int_{\partial \Omega} u\,g(X, \nu) \, \imath_{\partial\Omega}{}^*(\nu \lrcorner \omega_{\mathrm{vol}}) \,\, ,
\end{align*}
which concludes the proof.
\end{proof}

For the sake of notation, we denote by $\partial_\nu u$ the normal derivative
$$\partial_\nu u\coloneqq g(\nabla u,\nu) \,\, .$$
Moreover, we recall that the {\it second fundamental form of $\partial \Omega$} is the symmetric $(0,2)$-tensor field
\begin{equation}
\mathrm{II}: T\partial \Omega \otimes T\partial \Omega \to \mathbb{R} \,\, , \quad \mathrm{II}(X,Y) \coloneqq -g(D_XY,\nu) \quad \text{ for any } X, Y \in \mathcal{C}^{\infty}(\partial \Omega,T\partial \Omega) \,\, .
\end{equation}
For the sake of shortness, we introduce the following notation:

\begin{definition} \label{classO}
We denote by $\mathcal{O}$ the class of all the domains $\Omega \subset M$ with the following property: $\Omega$ has compact closure and the boundary $\partial \Omega$ is an orientable, embedded hypersurface of class $\mathcal{C}^{\infty}$. For any $\Omega \in \mathcal{O}$, we denote by $\nu$ the non-vanishing, unitary, normal, outward-pointing vector field on $\partial \Omega$ and by $\mathrm{II}$ the second fundamental form of $\partial \Omega$. Furthermore,  we define the subset
$$
\mathcal{O}^+ \coloneqq \big\{\Omega \in \mathcal{O} : \text{ the second fundamental form $\mathrm{II}$ of $\partial \Omega$ is non-negative definite}\big\}
$$
\end{definition}

We stress that a domain $\Omega \in \mathcal{O}$ belongs to the class $\mathcal{O}^+$ if and only if it is {\it locally geodetically convex}, i.e. for any $x \in \partial\Omega$ there exists a neighborhood $\mathcal{U}_x \subset M$ of $x$ such that the intersection $\mathcal{U}_x \cap \Omega$ is geodetically convex \cite{Bishop}. Moreover, if $(M,g)$ has nonpositive sectional curvature, e.g. the Euclidean space or the hyperbolic space, then the following strengthened version of this result holds true: a domain $\Omega \in \mathcal{O}$ belongs to the class $\mathcal{O}^+$ if and only if it is geodetically convex \cite[Theorem 1]{Alexander}. According to this fact, the following lemma is a direct generalization of \cite[Lemme I.1]{Lio80}.

\begin{lemma} \label{segno}
Let $\Omega \subset M$ be a domain in the class $\mathcal{O}$ and $u \in \mathcal{C}^2(\overline{\Omega},\mathbb{R})$ a function satisfying the Neumann condition $\partial_{\nu} u =0$ on $\partial \Omega$. If $w \coloneqq \tfrac12|\nabla u |^2$, then
$$
\partial_\nu w = -\mathrm{II}(\nabla u,\nabla u) \,\, .
$$
In particular, if $\Omega \in \mathcal{O}^+$, then $\partial_\nu w \leq 0$.
\end{lemma}

\begin{proof} Notice that the Neumann condition $\partial_{\nu} u =0$ is equivalent to the gradient $\nabla u$ being tangent to the boundary $\partial \Omega$. Then, by \eqref{gradw} and a straightforward computation
$$
\partial_\nu w = g(D^2u(\nabla u), \nu) = g(D_{\nabla u}\nabla u,\nu) = -\mathrm{II}(\nabla u,\nabla u)
$$
and so we get the thesis.
\end{proof}

\subsection{Sobolev spaces and inequalities} \hfill \par

Let $\Omega \subset M$ be a domain in the class $\mathcal{O}$ and $E \to M$ a vector bundle of tensors over $M$, e.g. the tangent bundle $TM \to M$ or the bundle of symmetric endomorphisms ${\rm Sym}(TM,g) \to M$. For any smooth tensor field $T \in \mathcal{C}^{\infty}(\Omega,E|_{\Omega})$ and for any $1 \leq p < \infty$, we set
$$
\| T \|_{L^p(\Omega,E|_{\Omega})} \coloneqq \left({\textstyle\int_{\Omega}} |T|^p \,\omega_{\mathrm{vol}} \right)^{\frac1p} \,\, .
$$
Furthermore, given $k \in \mathbb{Z}$, $k \geq0$, we define
$$
\| T \|_{W^{k,p}(\Omega,E|_{\Omega})} \coloneqq \sum_{j=0}^k \| D^j T \|_{L^p(\Omega,E^{(j)}|_{\Omega})} \,\, , \quad \text{ with } E^{(j)} \coloneqq (T^*M)^{\otimes j} \otimes E \,\, ,
$$
where $D^j T$ denotes the $j$-th covariant derivative of $T$. Then, we define the Sobolev space $W^{k,p}(\Omega,E|_{\Omega})$, resp. $W^{k,p}_0(\Omega,E|_{\Omega})$, as the completion of $\mathcal{C}^{\infty}(\overline{\Omega},E|_{\Omega})$, resp. $\mathcal{C}^{\infty}_{\rm c}(\Omega,E|_{\Omega})$, with respect to the norm $\| \cdot \|_{W^{k,p}(\Omega,E|_{\Omega})}$. As usual, we adopt the notation $L^p(\Omega,E|_{\Omega}) \coloneqq W^{0,p}(\Omega,E|_{\Omega})$. Here, $\mathcal{C}^{\infty}_{\rm c}(\Omega,E|_{\Omega})$ denotes the space of smooth sections of $E$ with compact support in $\Omega$. In the special case of the trivial line bundle $E=M \times \mathbb{R}$, we use the notations $L^p(\Omega,\mathbb{R})$ and $W^{k,p}(\Omega,\mathbb{R})$. \smallskip

Being $\overline{\Omega}$ compact, it holds that the definition of $W^{k,p}(\Omega,E|_{\Omega})$ does not depend on the Riemannian metric $g$ \cite[Proposition 2.2]{HebeyBook} and that $W^{k,q}(\Omega,E|_{\Omega}) \subset W^{k,p}(\Omega,E|_{\Omega}) \subset W^{k,1}(\Omega,E|_{\Omega})$ for any $1 \leq p \leq q < \infty$. Moreover, the following Sobolev embeddings hold (see \cite[Chapter 10]{HebeyBook}).

\begin{prop}[Sobolev Inequality]
Let $\Omega \subset M$ be a domain in the class $\mathcal{O}$. Then $W^{1,2}(\Omega,\mathbb{R}) \subset L^{\frac{2d}{d-2}}(\Omega,\mathbb{R})$, i.e. there exists a constant $C_{\rm Sob}(\Omega)>0$ such that
\begin{equation} \label{sob}
\|u\|_{L^{\frac{2d}{d-2}}(\Omega,\mathbb{R})}\leq C_{\rm Sob}(\Omega)\big(\|\nabla u\|_{L^2(\Omega,TM|_{\Omega})}+\|u\|_{L^2(\Omega,\mathbb{R})}\big)
\end{equation}
for any function $u\in W^{1,2}(\Omega,\mathbb{R})$.
\end{prop}

As in the Euclidean case \cite[Thm 5.22]{AdamsF}, the following extension result holds true for domains in the class $\mathcal{O}$. For the sake of completeness we sketch the proof, which is very similar to the Euclidean counterpart.

\begin{lemma} \label{lem:extSob}
Let $\Omega_1, \Omega_2 \subset M$ be two domains in the class $\mathcal{O}$ such that $\overline{\Omega}_1 \subset \Omega_2$. Then, there exist a linear operator
$$
T: \mathcal{C}^{\infty}\big(\overline{\Omega}_1,\mathbb{R}\big) \to \mathcal{C}^{\infty}_{\rm c}\big(\Omega_2,\mathbb{R}\big)
$$
which verifies the following properties: \begin{itemize}
\item[$a)$] for any $u \in \mathcal{C}^2\big(\overline{\Omega}_1,\mathbb{R}\big)$ it holds
$$
(Tu)|_{\overline{\Omega}_1} = u \,\, ;
$$
\item[$b)$] for any $k \geq0$ and $1 \leq p < \infty$, $T$ can be extended to a bounded linear operator
$$
T: W^{k,p}\big(\Omega_1,\mathbb{R}\big) \to W^{k,p}_0\big(\Omega_2,\mathbb{R}\big) \,\, .
$$
\item[$c)$] for any $1 \leq p < \infty$, there exists a constant $K_p>0$ such that
$$
\|\Delta (Tu)\|_{L^p(\Omega_2,\mathbb{R})} \leq K_p \, \|\Delta u\|_{L^p(\Omega_1,\mathbb{R})}
$$
for any $u \in \mathcal{C}^2\big(\overline{\Omega}_1,\mathbb{R}\big)$
\end{itemize}
\end{lemma}

\begin{proof}
In the following proof, we denote by $B(0,r) \coloneqq \{y \in \mathbb{R}^d : |y|<r \big\}$ the Euclidean ball in $\mathbb{R}^d$ of radius $r>0$ centered at the origin and by $\mathscr{B}_g(x,r)$ the open ball in $M$ of radius $r>0$ centered at the point $x \in M$ determined by the Riemannian distance $\mathtt{d}_g$ induced by $g$. Moreover, for any subset $E \subset \mathbb{R}^d$, we set
$$\begin{gathered}
E^{(>0)} \coloneqq \big\{y = (y^1,{\dots},y^d) \in E : y^d>0\big\} \,\, , \quad
E^{(=0)} \coloneqq \big\{y = (y^1,{\dots},y^d) \in E : y^d=0\big\} \,\, , \\
E^{(<0)} \coloneqq \big\{y = (y^1,{\dots},y^d) \in E : y^d<0\big\} \,\, , \quad 
E^{(\geq 0)} \coloneqq E^{(>0)} \cup E^{(=0)} \,\, .
\end{gathered}$$

By hypothesis, there exist $\epsilon>0$ and a finite sequence of points $\{x_1,{\dots},x_N\} \subset \partial \Omega_1$ on the boundary of $\Omega_1$ such that, setting $\mathcal{U}_i \coloneqq \mathscr{B}_g(x_i,\epsilon)$, the following properties hold true: \begin{itemize}
\item[$i)$] $\big\{\mathcal{U}_1,{\dots},\mathcal{U}_N\big\}$ is a finite open cover of $\partial \Omega_1$ and $\overline{\mathcal{U}}_s \subset \Omega_2$ for any $1 \leq s \leq N$;
\item[$ii)$] for any $1 \leq s \leq N$ there exists a smooth diffeomorphism $\psi_s : \mathcal{U}_s \to B(0,1)$ such that $\psi_s(x_s)=0$ and
$$
\psi_s(\mathcal{U}_s \cap \Omega_1) = B(0,1)^{(>0)} \,\, , \quad
\psi_s(\mathcal{U}_s \cap \partial \Omega_1) = B(0,1)^{(=0)} \,\, , \quad
\psi_s(\mathcal{U}_s \setminus \overline{\Omega}_1) = B(0,1)^{(<0)} \,\, ;
$$
\item[$iii)$] there exists $\delta>0$ such that the tubular neighborhood $\mathcal{N} \coloneqq \{x \in M : \mathtt{d}_g(x,\partial \Omega_1) < \delta\}$ verifies
$$
\mathcal{N} \subset \bigcup_{1 \leq s \leq N} \psi_s^{-1} \big(B(0,\tfrac12)\big) \,\, ;
$$
\item[$iv)$] the pulled-back metrics $g_s \coloneqq (\psi_s^{-1})^*g$ extends smoothly to $\overline{B(0,1)}$ and verifies
$$
(g_s)_{ij}(0) = \delta_{ij} \,\, , \quad \tfrac12\delta_{ij}\leq (g_s)_{ij}(y) \leq 2\delta_{ij} \,\, \text{ as bilinear forms for any $y \in B(0,1)$} \,\, ;
$$
\item[$v)$] for any $k \geq 0$ and $1 \leq p < \infty$, there exists $C_{k,p}>0$ such that
$$\begin{gathered}
\big\|\, (\psi_q|_{\mathcal{U}_s \cap \mathcal{U}_q}) \circ (\psi_s^{-1}|_{\psi_s(\mathcal{U}_s \cap \mathcal{U}_q)}) \,\big\|_{W^{k,p}(\psi_s(\mathcal{U}_s \cap \mathcal{U}_q))} \leq C_{k,p} \quad \text{for any $1 \leq s , q \leq N$ with $\mathcal{U}_s \cap \mathcal{U}_q \neq \emptyset$} \,\, , \\
\|(g_s)_{ij}\|_{W^{k,p}(B(0,1))} \leq C_{k,p} \quad \text{ for any } 1 \leq s \leq N \, , \,\, 1 \leq i,j \leq d \,\, . 
\end{gathered}$$
\end{itemize}
Let us consider the open set
$$
\mathcal{Q} \coloneqq \Big\{(y',y^d) \in \mathbb{R}^{d-1} \times \mathbb{R} : |y'| < \tfrac12 , |y^d| < \tfrac{\sqrt3}2 \Big\} \subset \mathbb{R}^d
$$
and observe that $B(0,\tfrac12) \subset \mathcal{Q} \subset B(0,1)$. Hence, the open sets $\mathcal{V}_k \coloneqq \psi_k^{-1}(\mathcal{Q}) \subset \mathcal{U}_k$ form a finite open cover of $\mathcal{N}$. We also choose another open set $\mathcal{V}_0 \subset \Omega_1$, bounded away from $\partial \Omega_1$, such that $\{\mathcal{V}_0,\mathcal{V}_1,{\dots},\mathcal{V}_N\}$ is an open cover of $\Omega_1$. We then take a smooth partition of unity $\{\eta_0,\eta_1,{\dots},\eta_N\}$ subordinate to $\{\mathcal{V}_0,\mathcal{V}_1,{\dots},\mathcal{V}_N\}$ and we assume that for any $k \geq 0$, $1 \leq p < \infty$ it holds
$$
\|\eta_s\|_{W^{k,p}(\Omega_2,\mathbb{R})} \leq C_{k,p} \quad \text{ for any } 1 \leq s \leq N \,\, .
$$
Fix $u \in \mathcal{C}^{\infty}\big(\overline{\Omega}_1,\mathbb{R}\big)$ and, for any $1 \leq s \leq N$, define
$$
u_s \coloneqq \big(\eta_s \circ (\psi_s^{-1}|_{B(0,1)^{(\geq0)}})\big) \cdot \big(u \circ (\psi_s^{-1}|_{B(0,1)^{(\geq0)}})\big) \,\, .$$
Since ${\rm supp}(u_s) \subset \mathcal{Q^{(\geq0)}}$, we can trivially extend it to a function $u_s \in \mathcal{C}^{\infty}\big((\mathbb{R}^d){}^{(\geq0)}, \mathbb{R}\big)$. By \cite[Thm 5.21]{AdamsF}, there exists a linear operator
$$
\tilde{T}: \mathcal{C}^{\infty}\big((\mathbb{R}^d){}^{(\geq0)}, \mathbb{R}\big) \to \mathcal{C}^{\infty}\big(\mathbb{R}^d,\mathbb{R}\big)
$$
such that, for any $v \in \mathcal{C}^{\infty}\big((\mathbb{R}^d){}^{(\geq0)}, \mathbb{R}\big)$ with ${\rm supp}(v) \subset \mathcal{Q^{(\geq0)}}$, the following properties hold: \begin{itemize}
\item[$\bcdot$] if $y=(y',y^d)$ with $y^d \geq 0$, then $(\tilde{T}v)(y) =v(y)$;
\item[$\bcdot$] ${\rm supp}(\tilde{T}v) \subset \mathcal{Q}$;
\item[$\bcdot$] $\| \tilde{T}v \|_{W^{k,p}(\mathcal{Q},\mathbb{R})} \leq C_{k,p} \, \| v \|_{W^{k,p}(\mathcal{Q}^{(>0)},\mathbb{R})}$ for any $k \geq 0$, $1 < p < \infty$;
\item[$\bcdot$] for any $1 < p <\infty$ there exists $\tilde{K}_p>0$ such that $\|\Delta (\tilde{T}v)\|_{L^p(\mathcal{Q},\mathbb{R})} \leq \tilde{K}_p \, \|\Delta v\|_{L^p(\mathcal{Q}^{(>0)},\mathbb{R})}$.
\end{itemize}
In the last two properties above, the Laplacian and the Sobolev norms in $ \mathcal{Q} \subset B(0,1)$ are taken with respect to the standard Euclidean metric.

Let us define now
$$
Tu \coloneqq (\eta_0 \cdot u) + \sum_{s=1}^N (\tilde{T}u_s) \circ \psi_s \,\, .
$$
Then, it follows by construction that $Tu \in \mathcal{C}^{\infty}_{\rm c}\big(\Omega_2,\mathbb{R}\big)$ and that $(Tu)(x)=u(x)$ for any $x \in \overline{\Omega}_1$. This gives rise to a linear operator $T: \mathcal{C}^{\infty}\big(\overline{\Omega}_1,\mathbb{R}\big) \to \mathcal{C}^{\infty}_{\rm c}\big(\Omega_2,\mathbb{R}\big)$ satisfying claim $(a)$. Moreover, a direct computation shows that for any $k \geq 0$ and $1 \leq p < \infty$, there exists a constant $L=L(k,p,d,C_{k,p},\Omega_1,\Omega_2)>0$ such that
$$
\| Tu \|_{W^{k,p}(\Omega_2,\mathbb{R})} \leq L \, \| u \|_{W^{k,p}(\Omega_1,\mathbb{R})} \,\, ,
$$
from which claim $(b)$ follows. Finally, a similar computation proves claim $(c)$.
\end{proof}

Therefore, from Lemma \ref{lem:extSob} and \cite[Thm 4.1]{GunPig}, the following Calder\'on-Zygmund inequality holds.

\begin{prop}[Calder\'on-Zygmund Inequality]\label{czineq}
Let $\Omega \subset M$ be a domain in the class $\mathcal{O}$. Then, for any $1 < p < \infty$ there exists $C_{\rm CZ}(\Omega,p)>0$ such that
\begin{equation}
\|D^2u\|_{L^p(\Omega,{\rm Sym}(TM,g)|_{\Omega})} \leq C_{\rm CZ}(\Omega,p) \, \|\Delta u\|_{L^p(\Omega,\mathbb{R})}
\end{equation}
for any function $u \in \mathcal{C}^{\infty}(\overline{\Omega},\mathbb{R})$.
\end{prop}

\medskip
\section{Main results of the paper} \label{sect_3}
\setcounter{equation} 0

Let $(M^d,g)$ be a Riemannian manifold and consider the Neumann problem
\begin{equation} \label{hjn}
\begin{cases}
-\Delta u+ \tfrac1{\gamma}|\nabla u|^\gamma+g(B,\nabla u) =f & \text{ in }\Omega\subset M \\
\partial_\nu u=0&\text{ on } \partial\Omega 
\end{cases}
\end{equation}
under the following assumptions on $\Omega$:
\begin{equation} \label{D1} \tag{D1}
\text{$\Omega \subset M$ is a domain in the class $\mathcal{O}^+$}\,\, ;
\end{equation}
\begin{equation} \label{D2} \tag{D2}
\begin{array}{c}
{\rm vol}(\Omega) \leq \rho \,\, , \\[0.5pt]
\text{$\mathrm{Ric}_x(\xi,\xi)\geq -\kappa |\xi|^2$ for any $x \in \Omega$, $\xi \in T_xM$}\, , \\
\text{the constant $C_{\rm Sob}(\Omega)$ in \eqref{sob} can be chosen such that $C_{\rm Sob}(\Omega)\leq\sigma$}\, ;
\end{array}
\end{equation}
and the following assumptions on the ingredients:
\begin{equation} \label{In1} \tag{In1}
\text{$f \in \mathcal{C}^1(\Omega,\mathbb{R}) \cap \mathcal{C}^0(\overline{\Omega},\mathbb{R})$}\,\, , \quad \text{$\gamma>1$} \,\, ;
\end{equation}
\begin{equation} \label{In2} \tag{In2}
\text{$B \in \mathcal{C}^1(\Omega,TM|_{\Omega}) \cap \mathcal{C}^0(\overline{\Omega},TM|_{\overline{\Omega}})$\, and there exists $s{>}d$ such that $\|B\|_{L^s(\Omega,\mathbb{R})} \leq \theta$} \,\, .
\end{equation}
Notice that in the conditions above $\kappa,\rho,\sigma,\theta$ denotes some given positive constants. \smallskip

It is well-known that in the Euclidean space $(\mathbb{R}^d,\langle\,,\rangle)$ the geometry of the domain plays a crucial role. In particular, it is shown in \cite[Lemme I.1]{Lio80} that when $\Omega$ is convex one has $\partial_\nu|\nabla u|^2\leq 0$. Then, after applying an integral Bernstein argument as in \cite{Lions85}, one gets the result, see \cite[Theorem 19]{CirantJMPA}. Here, we prove a more general version of \cite{CirantJMPA} assuming an appropriate counterpart of the above convexity assumption in the more general realm of Riemannian manifolds, see Lemma \ref{segno}, when the equation is perturbed with a first-order term having an unbounded coefficient. The gradient estimate is then achieved via a refined version of the Bochner's identity, see Lemma \ref{lemboc}, and adapting the integral technique in \cite{ll,Lions85}. Here, we stress that we work with classical solutions (or even strong solutions, cf \cite[Remark 3]{CGell} or \cite{BardiPerthame}), but bounds will depend only on the integrability properties of the data and, additionally, to a lower bound on the Ricci curvature coming from Bochner's formula, so this can be regarded as an a priori estimate. \smallskip

Our first main result reads as follows.

\begin{theorem} \label{main1}
Let $u \in \mathcal{C}^3(\Omega,\mathbb{R}) \cap \mathcal{C}^2(\overline{\Omega},\mathbb{R})$ be a solution of the Neumann problem \eqref{hjn} and assume that \eqref{D1}, \eqref{D2}, \eqref{In1}, \eqref{In2} hold true.
Then, there exists $\bar{r}=\bar{r}(d,\gamma)>1$ such that, for any $r>\bar{r}$, there exist a constant $C=C(d,\gamma,\kappa,\rho,\sigma,\theta,r,s)>0$ and an exponent $q=q(d,r)$, with $1 < q < d$, such that
\begin{equation} \label{thesis1}
\|\nabla u\|_{L^r(\Omega,TM|_{\Omega})}\leq C\big(1+\|f\|_{L^q(\Omega,\mathbb{R})}\big) \,\, , \quad \lim_{r \to +\infty}q(d,r) = d^{\,-} \,\, .
\end{equation}
\end{theorem}
As a byproduct of Theorem \ref{main1} one has for any $1 \leq r<\infty$ the existence of a constant $C>0$ such that 
$$
\|\nabla u\|_{L^r(\Omega,TM|_{\Omega})}\leq C\big(1+\|f\|_{L^d(\Omega,\mathbb{R})}\big) \,\, .
$$
This can be seen as nonlinear counterpart of \cite{CMjems} for the case $p=2$ and in the presence of lower-order terms with superlinear growth.
By exploiting the regularizing property of the equation we also have
\begin{equation} \label{boundgrad}
\text{ $f \in L^q(\Omega,\mathbb{R})$ for some $q>d$ } \quad \Longrightarrow \quad \nabla u \in L^\infty(\Omega,TM|_{\Omega}) \,\, .
\end{equation}
Indeed, if $f \in L^q(\Omega,\mathbb{R})$ for some $q>d$, then by \eqref{thesis1} it follows that $\nabla u \in L^r(\Omega,TM|_{\Omega})$ for all $r > \bar{r}(d,\gamma)$ and so, by using equation \eqref{hjn} together with condition \eqref{In2} and the H{\"o}lder Inequality, it follows that $\Delta u \in L^q(\Omega,\mathbb{R})$ as well. Therefore, the Calder\'on-Zygmund Inequality (see Proposition \ref{czineq}) and the Sobolev embedding $W^{1,q}(\Omega,TM|_{\Omega}) \hookrightarrow L^{\infty}(\Omega,TM|_{\Omega})$ (see \cite[Chapter 10]{HebeyBook}) allow to conclude the assertion \eqref{boundgrad}. We emphasize that the restriction $q>d$ is necessary to ensure $L^{\infty}$-gradient estimates even for linear equations in the Euclidean setting, see \cite{KMBull} and also \cite{CG2} for an explicit counterexample in the parabolic case via the fundamental solution of the heat equation. \smallskip

We now turn to maximal $L^q$-regularity properties for the equation \eqref{hjn}. Here, we assume
\begin{equation} \label{tildeIn2} \tag{$\widetilde{{\rm In}2}$}
B = 0
\end{equation}
instead of \eqref{In2} and we shall use the same scheme developed in \cite{CGell}. Nonetheless, compared to \cite{CGell}, here we face two additional difficulties. First, we treat the Neumann problem and hence handle boundary terms. Second, we have additional terms involving the Ricci curvature that can be absorbed via a suitable application of Young's inequalities. Our proof instead here builds upon a generalized Bochner's formula proved in Lemma \ref{lemboc}.

\begin{theorem} \label{main2}
Let $u \in \mathcal{C}^3(\Omega,\mathbb{R}) \cap \mathcal{C}^2(\overline{\Omega},\mathbb{R})$ be a solution of the Neumann problem \eqref{hjn} and assume that \eqref{D1}, \eqref{D2}, \eqref{In1}, \eqref{tildeIn2} hold true. Assume that there exist an exponent $q$ and a constant $K>0$ such that
$$
q > \max\big\{d\tfrac{\gamma-1}{\gamma},2\big\} \,\, , \quad
\|f\|_{L^q(\Omega,\mathbb{R})}+\|\nabla u\|_{L^1(\Omega,TM|_{\Omega})}\leq K \,\, .
$$
Then, there exists a constant $C=C(d,\gamma,\kappa,\rho,\sigma,q,K)>0$ such that
\begin{equation}
\|\Delta u\|_{L^q(\Omega,\mathbb{R})}+ \||\nabla u|^\gamma\|_{L^q(\Omega,\mathbb{R})}\leq C \,\, .
\end{equation}
\end{theorem}

We remark that, by exploiting again the Calder\'on-Zygmund Inequality (see Proposition \ref{czineq}) together with a bootstrapping argument, one can even obtain from Theorem \ref{main2} estimates on $\|D^2 u\|_{L^q(\Omega,{\rm Sym}(TM,g)|_{\Omega})}$.

\subsection{Proof of Theorem \ref{main1}} \label{sect_proof1} \hfill \par

The proof is accomplished in several steps and auxiliary results to control the terms coming from the integral Bochner's identity, see \eqref{intEQ} below. \smallskip

Let $u \in \mathcal{C}^3(\Omega,\mathbb{R}) \cap \mathcal{C}^2(\overline{\Omega},\mathbb{R})$ be a solution to the Neumann problem \eqref{hjn} and set $w\coloneqq\tfrac12|\nabla u|^2$. By \eqref{bw1}, we obtain the following integral Bochner's identity for solutions to \eqref{hjn}
\begin{multline} \label{intEQ}
-\int_\Omega \Delta w\,w^p\,\omega_{\mathrm{vol}}
+\int_\Omega |D^2u|^2\,w^p\,\omega_{\mathrm{vol}} = \\
=
-\int_\Omega |\nabla u|^{\gamma-2} g(\nabla u,\nabla w)\,w^p\,\omega_{\mathrm{vol}}
-\int_\Omega g(\nabla (g(B,\nabla u)), \nabla u)\,w^p\,\omega_{\mathrm{vol}} \\
-\int_\Omega {\rm Ric}(\nabla u,\nabla u)\,w^p\,\omega_{\mathrm{vol}} 
+\int_\Omega g(\nabla f,\nabla u)\,w^p\,\omega_{\mathrm{vol}}
\end{multline}
for any $p\geq1$. We are going now to estimate each term appearing in the equation \eqref{intEQ}. We begin with the terms on the left hand side.

\begin{lemma}
There exists a constant $C_1=C_1(d,\gamma)>0$ such that
\begin{align}
\label{est1}
-\int_\Omega \Delta w\,w^p\,\omega_{\mathrm{vol}} &\geq
\tfrac{2p}{\sigma^2(p+1)^2} \left(\int_\Omega w^\frac{(p+1)d}{d-2}\,\omega_{\mathrm{vol}}\right)^{\frac{d-2}d} -\tfrac{4p}{(p+1)^2}{\rm vol}(\Omega)^{\frac{\gamma-1}{p+\gamma}}\left(\int_\Omega w^{p+\gamma}\,\omega_{\mathrm{vol}}\right)^{\frac{p+1}{p+\gamma}} \,\, , \\
\label{est2}
\int_\Omega |D^2u|^2\,w^p\,\omega_{\mathrm{vol}} &\geq
\tfrac12 \int_\Omega |D^2u|^2\,w^p\,\omega_{\mathrm{vol}}
+\tfrac{2^{\gamma}-1}{2\gamma^2d}\int_\Omega w^{p+\gamma}\,\omega_{\mathrm{vol}}
-C_1\left(\int_\Omega |B|^2\,w^{p+1}\,\omega_{\mathrm{vol}}
+\int_\Omega f^2\,w^p\,\omega_{\mathrm{vol}}\right) \,\, .
\end{align}
\end{lemma}

\begin{proof}
First, by using the integration by parts formula \eqref{intpart}, Lemma \ref{segno} and the chain rule, we obtain $$\begin{aligned}
-\int_\Omega \Delta w\,w^p\,\omega_{\mathrm{vol}} &= -\int_\Omega \mathrm{div}(\nabla w)\,w^p\,\omega_{\mathrm{vol}} \\ 
&=  \int_\Omega g(\nabla w, \nabla(w^p))\,\omega_{\mathrm{vol}} - \int_{\partial \Omega} w^p\,g(\nabla w, \nu) \, \imath_{\partial\Omega}{}^*(\nu \!\lrcorner \omega_{\mathrm{vol}}) \\
&\geq p\int_\Omega w^{p-1}|\nabla w|^2\,\omega_{\mathrm{vol}} \\
&= \tfrac{4p}{(p+1)^2}\int_\Omega \big|\nabla w^{\frac{p+1}2}\big|^2\,\omega_{\mathrm{vol}} \,\, .
\end{aligned}$$
By the Sobolev's inequality \eqref{sob} applied to the function $w^{\frac{p+1}2}$, we get
$$
\tfrac1{2\sigma^2}\left(\int_\Omega w^\frac{(p+1)d}{d-2}\,\omega_{\mathrm{vol}}\right)^{\frac{d-2}d} \leq \int_\Omega \big|\nabla w^{\frac{p+1}2}\big|^2\,\omega_{\mathrm{vol}}+ \int_\Omega w^{p+1}\,\omega_{\mathrm{vol}}
$$
and so
\begin{equation} \label{est1'}
-\int_\Omega \Delta w\,w^p\,\omega_{\mathrm{vol}} \geq \tfrac{2p}{\sigma^2(p+1)^2} \left(\int_\Omega w^\frac{(p+1)d}{d-2}\,\omega_{\mathrm{vol}}\right)^{\frac{d-2}d} -\tfrac{4p}{(p+1)^2}\int_\Omega w^{p+1}\,\omega_{\mathrm{vol}} \,\, .
\end{equation}
Since ${\rm vol}(\Omega)<\infty$ and $\gamma>1$, it holds that
\begin{equation} \label{eqLpincl}
\int_\Omega w^{p+1}\,\omega_{\mathrm{vol}} \leq {\rm vol}(\Omega)^{\frac{\gamma-1}{p+\gamma}}\left(\int_\Omega w^{p+\gamma}\,\omega_{\mathrm{vol}}\right)^{\frac{p+1}{p+\gamma}}
\end{equation}
and so \eqref{est1} follows from \eqref{est1'} and \eqref{eqLpincl}.

We now handle the second estimate. By the Schwarz inequality, it is easy to see that
\begin{equation} \label{schw}
|D^2u|^2 \geq \tfrac1d (\Delta u)^2 \,\, .
\end{equation}
This observation allows to gain additional coercivity from the equation, and get
$$
|D^2u|^2 \geq \tfrac1d \big(\tfrac1{\gamma}|\nabla u|^\gamma +g(B,\nabla u)-f \big)^2 \,\, .
$$
Moreover, we recall that simple algebraic calculations lead to
\begin{equation} \label{ab}
(a+b-c)^2\geq a^2 -2a(|b|+|c|) \quad \text{ for any $a,b,c \in \mathbb{R}$ with $a\geq0$} \,\, .
\end{equation}
Then, by using \eqref{schw}, \eqref{ab} and by plugging the Hamilton-Jacobi equation \eqref{hjn} we obtain
\begin{equation} \label{est2'} \begin{aligned}
\int_\Omega |D^2u|^2\,w^p\,\omega_{\mathrm{vol}} &\geq \tfrac12\int_\Omega |D^2u|^2\,w^p\,\omega_{\mathrm{vol}} +\tfrac1{2d}\int_\Omega (\Delta u)^2\,w^p\,\omega_{\mathrm{vol}} \\
&\geq \tfrac12\int_\Omega |D^2u|^2\,w^p\,\omega_{\mathrm{vol}}
+\tfrac1{2\gamma^2d}\int_\Omega |\nabla u|^{2\gamma}\,w^p\,\omega_{\mathrm{vol}} -\tfrac1{\gamma d}\int_\Omega |B|\,|\nabla u|^{\gamma+1}w^p\,\omega_{\mathrm{vol}} \\
&\qquad \qquad -\tfrac1{\gamma d}\int_\Omega |f|\,|\nabla u|^{\gamma}w^p\,\omega_{\mathrm{vol}} \\
&= \tfrac12\int_\Omega |D^2u|^2\,w^p\,\omega_{\mathrm{vol}} +\tfrac{2^{\gamma-1}}{\gamma^2d}\int_\Omega w^{p+\gamma}\,\omega_{\mathrm{vol}}
-\tfrac{2^{\frac{\gamma+1}2}}{\gamma d} \int_\Omega |B|\,w^{p+\frac{\gamma+1}2}\,\omega_{\mathrm{vol}} \\
&\qquad \qquad -\tfrac{2^{\frac\gamma2}}{\gamma d} \int_\Omega |f|\,w^{p+\frac{\gamma}2}\,\omega_{\mathrm{vol}} \,\, .
\end{aligned} \end{equation}
By the generalized Young's inequality, it follows that there exist $C'_1,\tilde C''_1>0$ depending on $d,\gamma$ such that
\begin{equation} \label{est2''} \begin{gathered}
\tfrac{2^{\frac{\gamma+1}2}}{\gamma d} \int_\Omega |B|\,w^{p+\frac{\gamma+1}2}\,\omega_{\mathrm{vol}} \leq
\tfrac{1}{4\gamma^2d}\int_\Omega w^{p+\gamma}\,\omega_{\mathrm{vol}}
+ C'_1 \int_\Omega |B|^2\,w^{p+1}\,\omega_{\mathrm{vol}} \,\, , \\
\tfrac{2^{\frac\gamma2}}{\gamma d} \int_\Omega |f|\,w^{p+\frac{\gamma}2}\,\omega_{\mathrm{vol}} \leq
\tfrac{1}{4\gamma^2d}\int_\Omega w^{p+\gamma}\,\omega_{\mathrm{vol}}
+ C''_1 \int_\Omega f^2\,w^{p}\,\omega_{\mathrm{vol}}
\end{gathered} \end{equation}
and so \eqref{est2} follows from \eqref{est2'} and \eqref{est2''} by taking $C_1=\max\{ C'_1, C''_1\}$.
\end{proof}

We now estimate the first-order nonlinear terms. More precisely

\begin{lemma}
There exist constants $\tilde{C}=\tilde{C}(d,\gamma)>0$ and $C_2=C_2(d,p)>0$ such that
\begin{align}
\label{est3}
-\int_\Omega |\nabla u|^{\gamma-2} g(\nabla u,\nabla w)\,w^p\,\omega_{\mathrm{vol}} &\leq
\frac{\tilde{C}}{p+1} \bigg(\int_\Omega |D^2u|^2w^p\,\omega_{\mathrm{vol}}
+\int_\Omega w^{p+\gamma}\,\omega_{\mathrm{vol}}\bigg) \,\, , \\
\label{est4}
-\int_\Omega g(\nabla (g(B,\nabla u)), \nabla u)\,w^p\,\omega_{\mathrm{vol}} &\leq
\tfrac16 \int_\Omega |D^2u|^2w^p\,\omega_{\mathrm{vol}}
+C_2\int_\Omega |B|^2w^{p+1}\,\omega_{\mathrm{vol}} \,\, .
\end{align}
\end{lemma}

\begin{proof}
By using again the integration by parts formula \eqref{intpart}, the Young's inequality and the Neumann boundary condition, we obtain
\begin{equation} \label{est4'} \begin{aligned}
\int_\Omega |\nabla u|^{\gamma-2} g(\nabla u,\nabla w)\,w^p\,\omega_{\mathrm{vol}}
&= \tfrac1{p+1} \int_\Omega g\big(|\nabla u|^{\gamma-2}\nabla u ,\nabla(w^{p+1})\big)\,\omega_{\mathrm{vol}}\\
&= -\tfrac1{p+1} \int_\Omega {\rm div}(|\nabla u|^{\gamma-2}\nabla u)\,w^{p+1}\,\omega_{\mathrm{vol}}-\tfrac1{p+1} \int_\Omega |\nabla u|^{\gamma-2}w^{p+1}\partial_\nu u\, \imath_{\partial\Omega}{}^*(\nu \!\lrcorner \omega_{\mathrm{vol}}) \\
&= -\tfrac1{p+1} \int_\Omega g(\nabla|\nabla u|^{\gamma-2},\nabla u)\,w^{p+1}\,\omega_{\mathrm{vol}} -\tfrac1{p+1}\int_\Omega |\nabla u|^{\gamma-2}\Delta u\,w^{p+1}\,\omega_{\mathrm{vol}} \,\, .
\end{aligned} \end{equation}
Then, by \eqref{schw}, \eqref{est4'} and the Young inequality, we obtain
$$\begin{aligned}
-\int_\Omega |\nabla u|^{\gamma-2} g(\nabla u,\nabla w)\,w^p\,\omega_{\mathrm{vol}} &\leq \tfrac{|\gamma-2|}{p+1} \int_\Omega |D^2u||\nabla u|^{\gamma-2}w^{p+1}\,\omega_{\mathrm{vol}} +\tfrac{\sqrt{d}}{p+1}\int_\Omega |D^2u||\nabla u|^{\gamma-2}w^{p+1}\,\omega_{\mathrm{vol}} \\
&= 2^{\frac{\gamma}2-1}\tfrac{|\gamma-2|+\sqrt{d}}{p+1} \int_{\Omega}|D^2u|w^{\frac{p}{2}}w^{\frac{p+\gamma}2}\,\omega_{\mathrm{vol}} \\
&\leq \frac{\tilde{C}}{p+1} \bigg(\int_\Omega |D^2u|^2w^p\,\omega_{\mathrm{vol}}
+\int_\Omega w^{p+\gamma}\,\omega_{\mathrm{vol}}\bigg)
\end{aligned}$$
for some $\tilde{C}=\tilde{C}(d,\gamma)>0$ and so we get \eqref{est3}.

We now handle the second estimate. From \eqref{gradw}, \eqref{schw}, the integration by parts formula \eqref{intpart}, the Young inequality and the Neumann boundary condition, we get
$$\begin{aligned}
-\int_\Omega g(\nabla (g(B,\nabla u)), \nabla u)\,w^p\,\omega_{\mathrm{vol}} &= \int_\Omega g(B,\nabla u)\,{\rm div}(w^p \nabla u)\,\omega_{\mathrm{vol}} -\int_{\partial\Omega} g(B,\nabla u)\,w^p\partial_\nu u \, \imath_{\partial\Omega}{}^*(\nu \lrcorner \omega_{\mathrm{vol}}) \\
&= p\int_\Omega g(B,\nabla u)\,g(\nabla w,\nabla u)\,w^{p-1}\,\omega_{\mathrm{vol}} +\int_\Omega g(B,\nabla u)\,\Delta u\,w^p\,\omega_{\mathrm{vol}} \\
&\leq p\int_\Omega |B|\,|\nabla u|^3\,|D^2u|\,w^{p-1}\,\omega_{\mathrm{vol}} +\int_\Omega |B|\,|\nabla u|\,|\Delta u|\,w^p\,\omega_{\mathrm{vol}} \\
&\leq \big(2^{\frac32}p+\sqrt{2d}\big)\int_\Omega |B|\,w^{\frac{p+1}2}\,|D^2u|\,w^{\frac{p}2}\,\omega_{\mathrm{vol}} \\
&\leq \tfrac16 \int_\Omega |D^2u|^2w^p\,\omega_{\mathrm{vol}}+ C_2 \int_\Omega |B|^2w^{p+1}\,\omega_{\mathrm{vol}}
\end{aligned}$$
for some $C_2=C_2(d,p)>0$. Hence, \eqref{est4} follows.
\end{proof}

We estimate now the last terms.

\begin{lemma}
There exists a constant $C_3=C_3(d,p)>0$ such that
\begin{gather}
\label{est5}
-\int_\Omega \mathrm{Ric}(\nabla u,\nabla u)\,w^p\,\omega_{\mathrm{vol}} \leq
2\kappa {\rm vol}(\Omega)^{\frac{\gamma-1}{p+\gamma}} \bigg(\int_\Omega w^{p+\gamma}\,\omega_{\mathrm{vol}}\bigg)^{\frac{p+1}{p+\gamma}} \,\, , \\
\label{est6}
-\int_\Omega g(\nabla f,\nabla u)\,w^p\,\omega_{\mathrm{vol}} \leq \tfrac16\int_\Omega |D^2u|^2\,w^p\,\omega_{\mathrm{vol}} +C_3 \int_\Omega f^2\,w^p\,\omega_{\mathrm{vol}} \,\, . 
\end{gather}
\end{lemma}

\begin{proof}
By the lower bound on the Ricci tensor, we get
\begin{equation} \label{est5'}
-\int_\Omega \mathrm{Ric}(\nabla u,\nabla u)\,w^p\,\omega_{\mathrm{vol}} \leq 2\kappa \int_\Omega w^{p+1}\,\omega_{\mathrm{vol}}
\end{equation}
and so \eqref{est5} follows from \eqref{est5'} and \eqref{eqLpincl}.

We now handle the second estimate. By using integration by parts formula \eqref{intpart} and the Neumann boundary condition, it follows that
\begin{equation} \label{est3'} \begin{aligned}
-\int_\Omega g(\nabla f,\nabla u)\,w^p\,\omega_{\mathrm{vol}} &= \int_\Omega f\, \mathrm{div}(w^p\nabla u)\,\omega_{\mathrm{vol}}-\int_{\partial\Omega} f\,w^p\partial_\nu u\,\imath_{\partial\Omega}{}^*(\nu \lrcorner \omega_{\mathrm{vol}}) \\
&= \int_\Omega f\,g(\nabla(w^p), \nabla u)\,\omega_{\mathrm{vol}} +\int_\Omega f\,\Delta u\,w^p\,\omega_{\mathrm{vol}} \,\, .
\end{aligned} \end{equation}
Since a direct computation shows that
\begin{equation} \label{est3''}
\big|g(\nabla (w^p), \nabla u)\big| \leq 2p w^p |D^2u| \,\, ,
\end{equation}
from \eqref{schw}, \eqref{est3'}, \eqref{est3''} and the Young inequality, we get
$$\begin{aligned}
-\int_\Omega g(\nabla f,\nabla u)\,w^p\,\omega_{\mathrm{vol}} &\leq
\int_\Omega |f|\,\big|g(\nabla (w^p), \nabla u)\big|\,\omega_{\mathrm{vol}} +\int_\Omega |f|\,|\Delta u|\,w^p\,\omega_{\mathrm{vol}} \\
&\leq \big(\sqrt{d}+2p\big)\int_\Omega |D^2u|\,|f|\,w^p \omega_{\mathrm{vol}} \\
&\leq \tfrac16\int_\Omega |D^2u|^2\,w^p\,\omega_{\mathrm{vol}} +C_3 \int_\Omega f^2\,w^p\,\omega_{\mathrm{vol}} \,\, .
\end{aligned}$$
for some $C_3=C_3(d,p)>0$. Hence, \eqref{est6} follows.
\end{proof}

The last lemma that we need is the following

\begin{lemma}
Let $s>d$ and $\theta>0$ as in \eqref{In2}. Then, there exists $C_4=C_4(d,\gamma,\rho,\sigma,p,s)>0$ such that
\begin{equation} \label{estB}
C_2\int_\Omega |B|^2\,w^{p+1}\,\omega_{\mathrm{vol}}
\leq
\tfrac{p}{\sigma^2(p+1)^2}\left(\int_\Omega w^{(p+1)\frac{d}{d-2}}\,\omega_{\mathrm{vol}}\right)^{\frac{d-2}d} +\theta^{\frac{2s}{s-d}}C_4 \left(\int_\Omega w^{p+\gamma}\,\omega_{\mathrm{vol}}\right)^{\frac{p+1}{p+\gamma}} \,\, .
\end{equation}
\end{lemma}
\begin{proof}
By the H\"older's inequality
$$
\int_\Omega |B|^2\,w^{p+1}\,\omega_{\mathrm{vol}} \leq \left(\int_\Omega |B|^s\,\omega_{\mathrm{vol}}\right)^{\frac2s} \left(\int_\Omega w^{(p+1)\frac{s}{s-2}}\,\omega_{\mathrm{vol}}\right)^{\frac{s-2}s} \,\, .
$$
Owing to the fact that $s>d$, we have $1 < \tfrac{s}{s-2} < \tfrac{d}{d-2}$, and hence $p+1< (p+1)\tfrac{s}{s-2}<(p+1)\tfrac{d}{d-2}$, by interpolation we have $\|w\|_{L^{(p+1)\frac{s}{s-2}}}\leq \|w\|_{L^{(p+1)\frac{d}{d-2}}}^{t} \|w\|_{L^{p+1}}^{1-t}$ for $t=\frac{d}{s}\in(0,1)$, so that
$$
\left(\int_\Omega w^{(p+1)\frac{s}{s-2}}\,\omega_{\mathrm{vol}}\right)^{\frac{s-2}s} \leq \left(\int_\Omega w^{(p+1)\frac{d}{d-2}}\,\omega_{\mathrm{vol}}\right)^{\frac{d-2}s} \left(\int_\Omega w^{p+1}\,\omega_{\mathrm{vol}}\right)^{\frac{s-d}s} \,\, .
$$
Then, by using the Young Inequality with exponents $\tfrac{s}{d}, \tfrac{s}{s-d}$ we get
\begin{align}
C_2\int_\Omega |B|^2\,w^{p+1}\,\omega_{\mathrm{vol}}
&\leq 
\left(\int_\Omega w^{(p+1)\frac{d}{d-2}}\,\omega_{\mathrm{vol}}\right)^{\frac{d-2}s}
\cdot C_2
\left(\int_\Omega |B|^s\,\omega_{\mathrm{vol}}\right)^{\frac2s}
\left(\int_\Omega w^{p+1}\,\omega_{\mathrm{vol}}\right)^{\frac{s-d}s} \nonumber\\
& \label{estB'}
\leq
\tfrac{p}{\sigma^2(p+1)^2}\left(\int_\Omega w^{(p+1)\frac{d}{d-2}}\,\omega_{\mathrm{vol}}\right)^{\frac{d-2}d} +C'_4
\left(\int_\Omega |B|^s\,\omega_{\mathrm{vol}}\right)^{\frac2{s-d}}
\int_\Omega w^{p+1}\,\omega_{\mathrm{vol}}
\end{align}
for some $C'_4=C'_4(d,p,\sigma,s)>0$. Therefore, \eqref{estB} follows from \eqref{In2}, \eqref{eqLpincl} and \eqref{estB'}.
\end{proof}

We are now ready to complete the proof of our first main theorem.

\begin{proof}[Proof of Theorem \ref{main1}]
From now until the end of the proof, we will denote by $C_i$, $i \in \mathbb{N}$, constants that depend only on the data $(d,\gamma,\kappa,\rho,\sigma,\theta,p,s)$. On the other hand, we remark that the constant $\tilde{C}$ coming from \eqref{est3} depends only on the data $(d,\gamma)$.

Notice that, back to \eqref{intEQ} and owing to the previous estimates \eqref{est1}, \eqref{est2}, \eqref{est3}, \eqref{est4}, \eqref{est5}, \eqref{est6}, we end up with
\begin{multline} \label{final1}
\tfrac{p}{\sigma^2(p+1)^2}\left(\int_\Omega w^\frac{(p+1)d}{d-2}\,\omega_{\mathrm{vol}}\right)^{\frac{d-2}d}
+\tfrac16\int_{\Omega}|D^2u|^2w^{p} +\tfrac{2^{\gamma}-1}{2\gamma^2d} \int_\Omega w^{p+\gamma}\,\omega_{\mathrm{vol}}
\leq \\
\leq \frac{\tilde{C}(d,\gamma)}{p+1} \bigg(\int_\Omega |D^2u|^2w^p\,\omega_{\mathrm{vol}}
+\int_\Omega w^{p+\gamma}\,\omega_{\mathrm{vol}}\bigg)
+C_5\bigg(\left(\int_\Omega w^{p+\gamma}\,\omega_{\mathrm{vol}}\right)^{\frac{p+1}{p+\gamma}}
+\int_\Omega f^2\,w^p\,\omega_{\mathrm{vol}}\bigg) \,\, .
\end{multline}
We then apply the Young Inequality to the third term of the right-hand side of \eqref{final1} to get
$$
C_5\left(\int_\Omega w^{p+\gamma}\,\omega_{\mathrm{vol}}\right)^{\frac{p+1}{p+\gamma}}\leq \tfrac{2^{\gamma}-1}{4\gamma^2d} \int_\Omega w^{p+\gamma}\,\omega_{\mathrm{vol}}+C_6 \,\, .
$$
Henceforth, after absorbing the first term of the right-hand side on the left-hand side of \eqref{final1}, it follows that there exists $\bar{p}=\bar{p}(d,\gamma)>1$ such that \begin{equation} \label{final2}
2\left(\int_\Omega w^\frac{(p+1)d}{d-2}\,\omega_{\mathrm{vol}}\right)^{\frac{d-2}d} \leq C_7\bigg(1+\int_\Omega f^2\,w^p\,\omega_{\mathrm{vol}}\bigg)
\end{equation}
for any $p>\bar{p}$. We now handle the term involving the source $f$. We use the H\"older inequality, with exponents
$$
\beta_p'=\tfrac{(p+1)d}{(d-2)p} \quad \text{ and } \quad \beta_p \coloneqq \tfrac{(p+1)d}{d+2p} \,\, ,
$$
together with the the generalized Young inequality to obtain
\begin{equation} \label{holderfin1} \begin{aligned}
C_7 \int_\Omega f^2\,w^p\,\omega_{\mathrm{vol}} &\leq C_7 \bigg(\int_\Omega w^{\frac{(p+1)d}{d-2}}\,\omega_{\mathrm{vol}}\bigg)^{\frac{(d-2)p}{(p+1)d}} \bigg(\int_\Omega f^{2\beta_p}\,\omega_{\mathrm{vol}}\bigg)^{\frac1{\beta_p}} \\
&\leq \left(\int_\Omega w^\frac{(p+1)d}{d-2}\,\omega_{\mathrm{vol}}\right)^{\frac{d-2}d}+C_8\big(\|f\|_{L^{2\beta_p}(\Omega,\mathbb{R})}\big)^{2(p+1)} \,\, .
\end{aligned} \end{equation}
Let us set now
$$
r \coloneqq \tfrac{2(p+1)d}{d-2} \,\, , \quad q \coloneqq 2\beta_p \,\, .
$$
Then, by means of \eqref{final2} and \eqref{holderfin1}, it is straightforward to see that there exists $\bar{r}=\bar{r}(d,\gamma)>1$ such that, for any $r\geq \bar{r}$, there exist a constant $C=C(d,\gamma,\kappa,\rho,\sigma,\theta,r,s)>0$ such that
$$
\|\nabla u\|_{L^r(\Omega,TM|_{\Omega})}\leq C\big(1+\|f\|_{L^{q}(\Omega,\mathbb{R})}\big) \,\, .
$$
Moreover, we get $r \to +\infty$ and $q \to d^{\,-}$ as $p \to +\infty$. This concludes the proof.
\end{proof}

\subsection{Proof of Theorem \ref{main2}} \hfill \par

The proof of Theorem \ref{main2} is a consequence of the following Proposition \ref{mainprop2}. We follow the presentation given in \cite{CGell}: we first state, without proving, Proposition \ref{mainprop2} and we show how to apply it to prove Theorem \ref{main2}. Subsequently, we prove Proposition \ref{mainprop2}.

\subsubsection{Proof of Theorem \ref{main2}: Part I} \hfill \par

Let us define the function
$$
h: [0,+\infty) \to \mathbb{R} \,\, , \quad h(t) \coloneqq \tfrac{2}{1+\delta}(1+t)^{\frac{1+\delta}{2}} \quad \text{ with $0<\delta<1$ to be chosen later }
$$
and set
$$
w \coloneqq \tfrac12|\nabla u|^2 \,\, , \quad z \coloneqq h(w) \,\, .
$$
We will work on super-level sets of $z$, and so we define
$$
z_k \coloneqq (z-k)^+ \quad \text{ and } \quad \Omega_k\coloneqq\{x\in\Omega:z>k\} \quad \text{ for any } k \geq 0 \,\, .
$$

The main result of this section is the following

\begin{prop} \label{mainprop2}
There exists $\delta=\delta(d,\gamma,\kappa,q) \in (0,1)$ and a continuous function $\zeta: [0,+\infty) \to [0,+\infty)$, which depends on $d,\gamma,\kappa,q,K$, such that
$$
\lim_{t \to 0^+} \zeta(t) = 0^+
$$
and, for any $k \ge 1$, it holds that
\begin{equation}\label{mainest}
\bigg(\int_{\Omega_k} z_k^{\frac{q\gamma}{1+\delta}}\, \omega_{\mathrm{vol}}\bigg)^{\frac{d-2}d} \leq
\int_{\Omega_k} z_k^{\frac{q\gamma}{1+\delta}}\, \omega_{\mathrm{vol}}
+\zeta (\rm{vol}(\Omega_k)) \,\, .
\end{equation}
\end{prop}

\noindent that implies Theorem \ref{main2}. We show here how to use it to get the maximal regularity estimate. Notice that it is the same proof of \cite{CGell}, that we propose here for the convenience of the reader.

\begin{proof}[Proof of Theorem \ref{main2}]
Let us set
$$
y_k \coloneqq \int_{\Omega_k} z_k^{\frac{q\gamma}{1+\delta}}\, \omega_{\mathrm{vol}}
$$
and define the function
$$
\phi: [0,+\infty) \to \mathbb{R} \,\, , \quad \phi(y) \coloneqq y^{\frac{d-2}{d}}-y \,\, ,
$$
so that \eqref{mainest} is equivalent to
\begin{equation} \label{eq:yk}
\phi(y_k) \leq \zeta ({\rm vol}(\Omega_k)) \quad \text{ for any } k\geq1 \,\, .
\end{equation}
Notice that $\phi$ has a unique maximizer $y^* \coloneqq (\frac{d-2}{d})^{\frac{d}{2}}$, with corresponding value $\phi^* \coloneqq \phi(y^*) >0$. Moreover, for any $\bar{\zeta} \in [0,\phi^*)$, the equation $\phi(y)=\bar{\zeta}$ has two roots, that we denote by $y^{\pm}(\bar{\zeta})$, such that $0 \leq y^{-}(\bar{\zeta}) < y^* < y^{+}(\bar{\zeta}) \leq 1$. Since $\zeta(t) \to 0^+$ as $t \to 0^+$, there exists $t^*=t^*(d,\gamma,p,K)$ such that
\begin{equation} \label{eq:tstar}
0 < t^* < \phi^* \,\, , \quad \zeta(t)<t^* \,\, \text{ for any } 0\leq t<t^* \,\, .
\end{equation}
By the Chebyshev inequality \cite[p.6]{Bogachev}
$$
{\rm vol}(\Omega_k) = {\rm vol}\Big(\Big\{x \in \Omega : w(x)^{\frac12} > \big(\big(\tfrac{1+\delta}2k\big)^{\frac2{1+\delta}}-1\big)^{\frac12}\Big\}\Big) \leq \frac1{\big(\big(\tfrac{1+\delta}2k\big)^{\frac2{1+\delta}}-1\big)^{\frac12}}\int_{\Omega} w^{\frac12}\, \omega_{\mathrm{vol}}
$$
and so
\begin{equation} \label{eq:Cheb}
\|\nabla u\|_{L^1(\Omega,TM|_{\Omega})} < \sqrt{2}\big(\big(\tfrac{1+\delta}2k\big)^{\frac2{1+\delta}}-1\big)^{\frac12}t^* \quad \Longrightarrow \quad {\rm vol}(\Omega_k) < t^* \,\, .
\end{equation}
Set
$$
k^* \coloneqq \tfrac2{1+\delta}\Big(1+\tfrac12\big(\tfrac1{t^*}\|\nabla u\|_{L^1(\Omega,TM|_{\Omega})}\big)^2\Big)^{\frac{1+\delta}2}
$$
and notice that by \eqref{eq:yk}, \eqref{eq:tstar} and \eqref{eq:Cheb}
$$
\phi(y_k) < t^* \quad \text{ for any } k > k^* \,\, ,
$$
which provides the alternative $0 \leq y_k < y^*$ or $y^* < y_k \leq 1$ for any $k > k^*$. By the smoothness assumption on $u$, the function $k \mapsto y_k$ is continuous and it eventually vanishes for large values of $k$. Hence, we deduce that
$$
0 \leq y_k < y^* \quad \text{ for any } k>k^* \,\, .
$$
Standard properties of the Lebesgue spaces give
$$
\||\nabla u|^\gamma\|_{L^q(\Omega,\mathbb{R})}^{\frac{1+\delta}{\gamma}} =
\||\nabla u|^{1+\delta}\|_{L^{\frac{\gamma q}{1+\delta}}(\Omega,\mathbb{R})} \leq
2^{\frac{1+\delta}2} \|(1+w)^{\frac{1+\delta}2}\|_{L^{\frac{\gamma q}{1+\delta}}(\Omega,\mathbb{R})} =
2^{\frac{\delta-1}2}(1+\delta)\,\|z\|_{L^{\frac{\gamma q}{1+\delta}}(\Omega,\mathbb{R})} \,\, .
$$
Then
\begin{equation} \label{eqfinT2(1)}
\|z\|_{L^{\frac{q\gamma}{1+\delta}}(\Omega,\mathbb{R})} \leq
\|z_{k^*}+k^*\|_{L^{\frac{q\gamma}{1+\delta}}(\Omega,\mathbb{R})} \leq
\|z_{k^*}\|_{L^{\frac{q\gamma}{1+\delta}}(\Omega_{k^*},\mathbb{R})}+k^*{\rm vol}(\Omega)^{\frac{q\gamma}{1+\delta}} \leq (y^*)^{\frac{1+\delta}{q\gamma}} +k^*{\rm vol}(\Omega)^{\frac{q\gamma}{1+\delta}}
\end{equation}
from which it follows that there exists $C=C(d,\gamma,\kappa,\rho,\sigma,q,K)>0$ such that
\begin{equation} \label{eqfinT2(2)}
\||\nabla u|^\gamma\|_{L^q(\Omega,\mathbb{R})} < C \,\, .
\end{equation}
The estimate on $\|\Delta u\|_{L^q(\Omega,\mathbb{R})}$ follows readily by using the equation itself and \eqref{eqfinT2(2)}.
\end{proof}

\subsubsection{Proof of Theorem \ref{main2}: Part II} \hfill \par

We proceed now with the Proof of Proposition \ref{mainprop2}. It is easy to verify that for any $t\geq0$
\begin{equation}\label{h1}
h'(t)t^{\frac12} \leq (1+t)^{\frac\delta2}\ ,
\end{equation}
while for $t\geq0$ and $\delta\in(0,1)$ straightforward computations lead to
\begin{equation}\label{h2}
h'(t) + 2th''(t) \geq \delta h'(t)\ .
\end{equation}
Moreover, for the sake of notation, we set
$$
z^{(1)} \coloneqq h'(w) \quad \text{ and } z^{(2)} \coloneqq h''(w) \,\, .
$$
and we observe that
\begin{equation} \label{h3}
z^{(1)} = (1+w)^{\frac{\delta-1}{2}}=\Big(\tfrac{\delta+1}{2}z \Big)^{\frac{\delta-1}{1+\delta}} \,\, .
\end{equation}
Notice that, by the regularity of $u$, it holds $z_k \in W^{1,\infty}(\Omega,\mathbb{R})$. Since by \eqref{gradw} and the chain rule it hold
$$
\nabla \big(\tfrac1{\gamma}|\nabla u|^{\gamma}\big) = |\nabla u|^{\gamma-2}D^2u(\nabla u) \quad \text{ and } \quad \nabla z = z^{(1)} D^2u(\nabla u) \,\, ,
$$
by
\eqref{bw2}, \eqref{hjn} and \eqref{tildeIn2} we get the following equation satisfied by $z$:
\begin{equation}\label{eqz}
-\Delta z +z^{(1)}|D^2u|^2 +z^{(2)}|D^2u(\nabla u)|^2 +z^{(1)}\mathrm{Ric}(\nabla u,\nabla u) +|\nabla u|^{\gamma-2}g(\nabla z,\nabla u) = z^{(1)}g(\nabla f,\nabla u) \,\, .
\end{equation}
We pick $\beta>1$ to be chosen later and we multiply \eqref{eqz} by $z_k^\beta$. Since $z_k$ vanishes on $\Omega\backslash\Omega_k$, integrating over $\Omega$ and $\Omega_k$ is the same, so we get
\begin{multline} \label{intEQ2}
-\int_{\Omega_k} \Delta z \, z_k^{\beta} \, \omega_{\mathrm{vol}}
+\int_{\Omega_k} \big(z^{(1)} \, |D^2u|^2  +z^{(2)} \, |D^2u(\nabla u)|^2\big) \, z_k^{\beta} \, \omega_{\mathrm{vol}} = \\
= -\int_{\Omega_k} |\nabla u|^{\gamma-2} \, g(\nabla z,\nabla u) \, z_k^{\beta} \, \omega_{\mathrm{vol}}
-\int_{\Omega_k} z^{(1)} \, \mathrm{Ric}(\nabla u,\nabla u) \, z_k^{\beta} \, \omega_{\mathrm{vol}}
+\int_{\Omega_k} z^{(1)} \, g(\nabla f,\nabla u) \, z_k^{\beta} \, \omega_{\mathrm{vol}} \,\, .
\end{multline}

\begin{lemma}
The following integral inequality holds:
\begin{multline} \label{mainineq}
\beta \int_{\Omega_k}z_k^{\beta-1}|\nabla z|^2\, \omega_{\mathrm{vol}}
+\tfrac{\delta}2 \int_{\Omega_k} z^{(1)}|D^2u|^2z_k^{\beta} \, \omega_{\mathrm{vol}}
+\Phi \int_{\Omega_k} z^{\eta} \, z_k^{\beta}\, \omega_{\mathrm{vol}} \leq \\
\leq \tfrac{\delta c(\gamma)}{d} \int_{\Omega_k} z^{(1)}\big(1+f^2\big) \, z_k^{\beta}\, \omega_{\mathrm{vol}}
-\int_{\Omega_k} fz^{(1)}z_k^{\beta}\Delta u\, \omega_{\mathrm{vol}}
-\int_{\Omega_k} fz^{(2)}z_k^{\beta}g(D^2u(\nabla u),\nabla u)\, \omega_{\mathrm{vol}} \\
-\beta \int_{\Omega_k} f z^{(1)}z_k^{\beta-1}g(\nabla z,\nabla u)\, \omega_{\mathrm{vol}}
-\int_{\Omega_k} |\nabla u|^{\gamma-2} \, g(\nabla z,\nabla u) \, z_k^{\beta} \, \omega_{\mathrm{vol}} 
-\int_{\Omega_k} z^{(1)} \, \mathrm{Ric}(\nabla u,\nabla u) \, z_k^{\beta} \, \omega_{\mathrm{vol}} \,\, ,
\end{multline}
where $c(\gamma) \coloneqq \max\big\{1,\tfrac{2^{\gamma-2}}{\gamma^2}\big\}$, $\Phi \coloneqq \tfrac{\delta}{2\gamma^2 d} \big(\tfrac{\delta+1}2\big)^{\frac{2\gamma+\delta-1}{\delta+1}}$ and $\eta \coloneqq \frac{2\gamma+\delta-1}{1+\delta}$.
\end{lemma}

\begin{proof}
As for the first term in \eqref{intEQ2}, notice that
\begin{equation} \label{cond_zk}
\begin{array}{rl}
\nabla z_k = \nabla z \,\,& \text{ in }\, \Omega_k \,\, , \\
z_k = 0 \,\,& \text{ on }\, \partial\Omega_k\cap \Omega \,\, , \\
\partial_{\nu}z = z^{(1)}\partial_{\nu}w \leq 0 \,\,& \text{ on }\, \partial\Omega_k\cap \partial \Omega \,\, ,
\end{array}
\end{equation}
where the last inequality holds because of Lemma \ref{segno} and the fact that $z^{(1)}\geq0$. Hence by \eqref{intpart} and \eqref{cond_zk} we get
\begin{equation} \label{intEQ2(1)}
-\int_{\Omega_k} \Delta z \, z_k^{\beta} \, \omega_{\mathrm{vol}} =
\int_{\Omega_k} g(\nabla (z_k^{\beta}), \nabla z) \, \omega_{\mathrm{vol}}
-\int_{\partial \Omega_k \cap \partial \Omega}z_k^\beta \partial_\nu z \, \imath_{\partial\Omega}{}^*(\nu \lrcorner \omega_{\mathrm{vol}})
\geq \beta \int_{\Omega_k}z_k^{\beta-1}|\nabla z|^2\, \omega_{\mathrm{vol}} \,\, .
\end{equation}

As for the second term in \eqref{intEQ2}, notice that
\begin{equation} \label{ab'}
(a-b)^2\geq \tfrac{a^2}{2}-2b^2 \quad \text{ for any $a,b \in \mathbb{R}$} \,\, .
\end{equation}
Then by the Cauchy-Schwarz inequality, \eqref{hjn}, \eqref{tildeIn2} and \eqref{ab'} we get \begin{equation} \label{cs1}
|D^2u|^2 \geq \tfrac1d(\Delta u)^2 \geq \tfrac1{2\gamma^2d}|\nabla u|^{2\gamma} -\tfrac2d f^2 \,\, .
\end{equation}
Moreover, since $h''(t)<0$ for any $t \geq 0$, from \eqref{h2} we get \begin{equation} \label{h''<0}
z^{(1)}|D^2u|^2 +z^{(2)}|D^2u(\nabla u)|^2 \geq \big(h'(w) +2wh''(w)\big)|D^2u|^2\geq \delta z^{(1)}|D^2u|^2
\end{equation}
and therefore, by \eqref{cs1} and \eqref{h''<0}
\begin{align}
z^{(1)}|D^2u|^2 +z^{(2)}|D^2u(\nabla u)|^2 &\geq \delta z^{(1)}|D^2u|^2 \nonumber \\
&\geq \tfrac{\delta}2 z^{(1)}|D^2u|^2 +\tfrac{\delta}{4\gamma^2d}z^{(1)}|\nabla u|^{2\gamma} -\tfrac{\delta}d z^{(1)} f^2 \,\, . \label{so}
\end{align}
Note that, since $\gamma>1$ and $w\geq0$, one has $(1+w)^{\gamma} \leq 2^{\gamma-1}(1+w^{\gamma})$ and so
$$
w^{\gamma} \geq 2^{1-\gamma}(1+w)^{\gamma} -1 \,\, ,
$$
which implies, together with \eqref{h3}, that
\begin{align}
\tfrac{\delta}{4\gamma^2d}z^{(1)}|\nabla u|^{2\gamma}&= \tfrac{\delta}{\gamma^2 d}2^{\gamma-2}(1+w)^{\frac{\delta-1}2}w^{\gamma} \nonumber \\
&\geq \tfrac{\delta}{2\gamma^2 d}(1+w)^{\frac{\delta-1}2+\gamma} -\tfrac{\delta}{\gamma^2 d}2^{\gamma-2}(1+w)^{\frac{\delta-1}2} \nonumber \\
&= \tfrac{\delta}{2\gamma^2 d} \big(\tfrac{\delta+1}2\big)^{\frac{2\gamma+\delta-1}{\delta+1}}z^{\frac{2\gamma+\delta-1}{\delta+1}} -\tfrac{\delta}{\gamma^2 d}2^{\gamma-2}z^{(1)} \,\, . \label{so'}
\end{align}
Set $c(\gamma) \coloneqq \max\big\{1,\tfrac{2^{\gamma-2}}{\gamma^2}\big\}$, $\Phi \coloneqq \tfrac{\delta}{2\gamma^2 d} \big(\tfrac{\delta+1}2\big)^{\frac{2\gamma+\delta-1}{\delta+1}}$ and $\eta \coloneqq \frac{2\gamma+\delta-1}{1+\delta}$. As a byproduct of \eqref{so} and \eqref{so'}, it holds that 
\begin{multline} \label{intEQ2(2)}
\int_{\Omega_k} \big(z^{(1)} |D^2u|^2  +z^{(2)} |D^2u(\nabla u)|^2\big) \, z_k^{\beta} \, \omega_{\mathrm{vol}} \geq \\
\geq \tfrac{\delta}2 \int_{\Omega_k} z^{(1)}|D^2u|^2z_k^{\beta} \, \omega_{\mathrm{vol}}
+\Phi \int_{\Omega_k} z^{\eta} \, z_k^{\beta}\, \omega_{\mathrm{vol}}
-\tfrac{\delta c(\gamma)}{d} \int_{\Omega_k} z^{(1)}\big(1+f^2\big) \, z_k^{\beta}\, \omega_{\mathrm{vol}} \,\, .
\end{multline}

As for the last term in \eqref{intEQ2}, by \eqref{intpart}, \eqref{cond_zk} and the Neumann boundary condition we get
\begin{align}
\int_{\Omega_k}& z^{(1)} \, g(\nabla f,\nabla u) \, z_k^{\beta} \, \omega_{\mathrm{vol}} = \nonumber \\
&= -\int_{\Omega_k} f\,{\rm{div}}(z^{(1)}z_k^{\beta}\nabla u) \, \omega_{\mathrm{vol}} +\int_{\partial \Omega_k \cap \partial \Omega} fz^{(1)}z_k^\beta\,\partial_\nu u \, \imath_{\partial\Omega}{}^*(\nu \lrcorner \omega_{\mathrm{vol}})\nonumber \\
&= -\beta \int_{\Omega_k} f z^{(1)}z_k^{\beta-1}g(\nabla z,\nabla u)\, \omega_{\mathrm{vol}} -\int_{\Omega_k} fz^{(2)}z_k^{\beta}g(D^2u(\nabla u),\nabla u)\, \omega_{\mathrm{vol}} -\int_{\Omega_k} fz^{(1)}z_k^{\beta}\Delta u\, \omega_{\mathrm{vol}} \,\, . \label{intEQ2(3)}
\end{align}

Hence, \eqref{mainineq} follows from \eqref{intEQ2}, \eqref{intEQ2(1)}, \eqref{intEQ2(2)} and \eqref{intEQ2(3)}.
\end{proof}

We now estimate all the terms on the right-hand side of \eqref{mainineq}. We start with the one involving the Ricci curvature.

\begin{lemma}
There exists $C_1 = C_1(d,\gamma,\delta,\kappa)>0$ such that
\begin{equation} \label{mainineq1}
-\int_{\Omega_k} z^{(1)} \, \mathrm{Ric}(\nabla u,\nabla u) \, z_k^{\beta} \, \omega_{\mathrm{vol}} \leq \tfrac14\Phi \int_{\Omega_k} z^{\eta}\,z_k^{\beta}\, \omega_{\mathrm{vol}}  +C_1 \int_{\Omega_k} z_k^{\beta}\, \omega_{\mathrm{vol}} \,\, .
\end{equation}
\end{lemma}

\begin{proof}
By the lower bound on the Ricci tensor and by \eqref{h3}, we have
\begin{align}
-\int_{\Omega_k} z^{(1)} \, \mathrm{Ric}(\nabla u,\nabla u) \, z_k^{\beta} \, \omega_{\mathrm{vol}} & \leq 2\kappa \int_{\Omega_k} z^{(1)} z_k^{\beta}(1+w)\, \omega_{\mathrm{vol}} \nonumber \\
& = 2\kappa \int_{\Omega_k} z_k^{\beta}(1+w)^{\frac{\delta+1}2}\, \omega_{\mathrm{vol}} \nonumber \\
& = \kappa(1+\delta) \int_{\Omega_k}z \, z_k^{\beta}\, \omega_{\mathrm{vol}} \,\, . \label{maineq1-1}
\end{align}
Notice now that, being $\gamma>1$, it holds that $\frac{2\gamma+\delta-1}{1+\delta}>1$ and so, by the weighted Young's inequality with the pair of exponents $(\eta,\tfrac{\eta}{\eta-1})$, we get
\begin{equation} \label{maineq1-2}
z \, z_k^{\beta} = z \, z_k^{\frac{\beta}{\eta}} \, z_k^{\beta(1-\frac1{\eta})} \leq \tfrac1{4\kappa(1+\delta)}\Phi\, z^{\eta} \, z_k^{\beta} + C_1' \,z_k^{\beta}
\end{equation}
for some $C'_1 = C'_1(d,\delta,\gamma,\kappa)>0$. Therefore, \eqref{mainineq1} follows from \eqref{maineq1-1} and \eqref{maineq1-2}.
\end{proof}

We now handle the first three terms on the right-hand side of \eqref{mainineq}. We now set
\begin{equation}\label{p}
p \coloneqq \tfrac{2}{d}\tfrac{d(\gamma-1)}{\gamma} +\tfrac{d-2}{d}q \quad \text{ and } \quad \beta \coloneqq \tfrac{\gamma(p-2)+1-\delta}{1+\delta} \,\, .
\end{equation}
From now on, we will assume that $p>2$, i.e. that $\gamma q > \frac{2d}{d-2}$. For the case $p\leq 2$, we postpone the reader to the comments given at the end of the proof.

\begin{lemma}
There exists $C_2 = C_2(d,\gamma,\delta)>0$ such that
\begin{multline} \label{mainineq2}
\tfrac{\delta c(\gamma)}{d} \int_{\Omega_k} z^{(1)}\big(1+f^2\big) \, z_k^{\beta}\, \omega_{\mathrm{vol}}
-\int_{\Omega_k} fz^{(1)}z_k^{\beta}\Delta u\, \omega_{\mathrm{vol}}
-\int_{\Omega_k} fz^{(2)}z_k^{\beta}g(D^2u(\nabla u),\nabla u)\, \omega_{\mathrm{vol}} \leq \\
\leq
\tfrac{\delta}2 \int_{\Omega_k}|D^2u|^2 z^{(1)} z_k^{\beta}\, \omega_{\mathrm{vol}}
+\tfrac14\Phi \int_{\Omega_k} z^{\eta}\,z_k^{\beta} \, \omega_{\mathrm{vol}}
+C_2 \int_{\Omega_k} (1+|f|^p) \, \omega_{\mathrm{vol}} \,\, .
\end{multline}
\end{lemma}

\begin{proof}
We use Cauchy-Schwarz inequality and the fact that $2th''(t)\leq h'(t)$ to obtain
\begin{multline*}
\tfrac{\delta c(\gamma)}{d} \int_{\Omega_k} z^{(1)}\big(1+f^2\big) \, z_k^{\beta}\, \omega_{\mathrm{vol}}
-\int_{\Omega_k} fz^{(1)}z_k^{\beta}\Delta u\, \omega_{\mathrm{vol}}
-\int_{\Omega_k} fz^{(2)}z_k^{\beta}g(D^2u(\nabla u),\nabla u)\, \omega_{\mathrm{vol}} \leq \\
\leq
\tfrac{\delta c(\gamma)}{d} \int_{\Omega_k} z^{(1)}\big(1+f^2\big) \, z_k^{\beta}\, \omega_{\mathrm{vol}}
+\big(2+\sqrt{d}\big) \int_{\Omega_k}|f|z^{(1)}z_k^{\beta} |D^2u| \, \omega_{\mathrm{vol}} \,\, .
\end{multline*}
Moreover, by the Young's inequality
$$
|f| z^{(1)} z_k^{\beta} |D^2u| \leq
\tfrac{\delta}{2(2+\sqrt{d})}|D^2u|^2 z^{(1)} z_k^{\beta}
+C_2'f^2 z^{(1)} z_k^{\beta}  \,\, ,
$$
for some $C_2'=C_2'(d,\delta)>0$, so that
\begin{multline} \label{1+f2}
\tfrac{\delta c(\gamma)}{d} \int_{\Omega_k} z^{(1)}\big(1+f^2\big) \, z_k^{\beta}\, \omega_{\mathrm{vol}}
-\int_{\Omega_k} fz^{(1)}z_k^{\beta}\Delta u\, \omega_{\mathrm{vol}}
-\int_{\Omega_k} fz^{(2)}z_k^{\beta}g(D^2u(\nabla u),\nabla u)\, \omega_{\mathrm{vol}} \leq \\
\leq
\tfrac{\delta}2 \int_{\Omega_k}|D^2u|^2 z^{(1)} z_k^{\beta}\, \omega_{\mathrm{vol}}
+C_2'' \int_{\Omega_k} z^{(1)}\big(1+f^2\big) \, z_k^{\beta}\, \omega_{\mathrm{vol}} \,\, ,
\end{multline}
with $C_2'' \coloneqq \big(\tfrac{\delta c(\gamma)}{d}+\big(2+\sqrt{d}\big)C_2'\big)$. Since $q>\frac{d(\gamma-1)}{\gamma}$ by hypothesis, in view of \eqref{p} it is easy to check that $\frac{d(\gamma-1)}{\gamma}<p<q$. Owing to these choices, we have $\beta>1$ whenever $\delta$ is sufficiently close to 0 and
\begin{gather}
\eta = \tfrac{2\gamma+\delta-1}{1+\delta}=\tfrac{\delta-1}{1+\delta}\tfrac{p}{p-2}+\beta\tfrac{2}{p-2} \,\, , \label{bo1} \\
(\beta+1)\tfrac{d}{d-2}=\tfrac{\gamma q}{1+\delta} \,\, . \label{bo2}
\end{gather}
Going back to \eqref{1+f2}, we use the H\"older inequality, the Young inequality, equations \eqref{h3}, \eqref{bo1} and the fact that $z_k\leq z$ to obtain
\begin{equation} \label{HY} \begin{aligned}
C_2'' \int_{\Omega_k} z^{(1)}\big(1+f^2\big) \, z_k^{\beta}\, \omega_{\mathrm{vol}} &= C_2'' \int_{\Omega_k}(1+f^2) z^{\frac{\delta-1}{\delta+1}} z_k^{\beta}\, \omega_{\mathrm{vol}} \\
&\leq C_2''
\bigg(\int_{\Omega_k} z^{\frac{\delta-1}{\delta+1}\frac{p}{p-2}}z_k^{\beta\frac{p}{p-2}}\, \omega_{\mathrm{vol}}\bigg)^{\frac{p-2}p}
\bigg(\int_{\Omega_k} (1+f^2)^{\frac{p}2} \, \omega_{\mathrm{vol}}\bigg)^{\frac2p} \\
&\leq \tfrac14\Phi \int_{\Omega_k} z^{\eta}\,z_k^{\beta} \, \omega_{\mathrm{vol}}
+C_2 \int_{\Omega_k} (1+|f|^p) \, \omega_{\mathrm{vol}} 
\end{aligned} \end{equation}
for some $C_2 = C_2(d,\gamma,\delta)>0$. Therefore, \eqref{mainineq2} follows from \eqref{1+f2} and \eqref{HY}.
\end{proof}

\begin{lemma}
There exists $C_3 = C_3(d,\gamma,\delta,q)$ such that
\begin{multline} \label{mainineq3}
-\beta \int_{\Omega_k} f z^{(1)}z_k^{\beta-1}g(\nabla z,\nabla u)\, \omega_{\mathrm{vol}} \leq \\
\leq \tfrac13\beta \int_{\Omega_k} |\nabla z|^2z_k^{\beta-1} \, \omega_{\mathrm{vol}}
+\tfrac14\Phi \int_{\Omega_k}z^{\eta}\, z_k^\beta\, \omega_{\mathrm{vol}}  +C_3(\|f\|_{L^q(\Omega,\mathbb{R})})^p \bigg(\int_{\Omega_k} (1+w)^{\frac{\delta}2\frac{pq}{q-p}} \, \omega_{\mathrm{vol}}\bigg)^{\frac{q-p}{q}} \,\, .
\end{multline}
\end{lemma}

\begin{proof}
We use now the inequality \eqref{h1}, and combine the H\"older, Young and the Cauchy-Schwarz inequalities to get
$$\begin{aligned}
-\beta &\int_{\Omega_k} f z^{(1)}z_k^{\beta-1}g(\nabla z,\nabla u)\, \omega_{\mathrm{vol}} \leq \\
&\leq \sqrt{2}\beta \int_{\Omega_k} |f| (1+w)^{\frac{\delta}2}|\nabla z| z_k^{\beta-1} \, \omega_{\mathrm{vol}} \\
& \leq \sqrt{2}\beta
\bigg(\int_{\Omega_k} |f|^q \, \omega_{\mathrm{vol}}\bigg)^{\frac1q}
\bigg(\int_{\Omega_k} (1+w)^{\frac{\delta}2\frac{pq}{q-p}} \, \omega_{\mathrm{vol}}\bigg)^{\frac{q-p}{pq}}
\bigg(\int_{\Omega_k} |\nabla z|^2z_k^{\beta-1} \, \omega_{\mathrm{vol}}\bigg)^{\frac12}
\bigg(\int_{\Omega_k} z_k^{(\beta-1)\frac{p}{p-2}} \, \omega_{\mathrm{vol}}\bigg)^{\frac{p-2}{2p}} \\
& \leq \tfrac13\beta \int_{\Omega_k} |\nabla z|^2z_k^{\beta-1} \, \omega_{\mathrm{vol}}
+\tfrac14\Phi \int_{\Omega_k} z_k^{(\beta-1)\frac{p}{p-2}} \, \omega_{\mathrm{vol}} \\
& \qquad\qquad\qquad +C_3 \bigg(\int_{\Omega_k} |f|^q \, \omega_{\mathrm{vol}}\bigg)^{\frac{p}q}
\bigg(\int_{\Omega_k} (1+w)^{\frac{\delta}2\frac{pq}{q-p}} \, \omega_{\mathrm{vol}}\bigg)^{\frac{q-p}{q}}
\end{aligned}$$
for some $C_3 = C_3(d,\gamma,\delta,q)$. Since $z> k \geq1$ on $\Omega_k$ and $z_k \leq z$, by \eqref{bo1} we get
$$
\int_{\Omega_k} z_k^{(\beta-1)\frac{p}{p-2}}\, \omega_{\mathrm{vol}} = 
\int_{\Omega_k} z_k^{\beta\frac{2}{p-2}-\frac{p}{p-2}} z_k^\beta\, \omega_{\mathrm{vol}}
\leq \int_{\Omega_k}z^{\beta\frac{2}{p-2}-\frac{1-\delta}{1+\delta}\frac{p}{p-2}}z_k^\beta\, \omega_{\mathrm{vol}} = \int_{\Omega_k}z^{\eta}\,z_k^\beta\, \omega_{\mathrm{vol}}
$$
and then \eqref{mainineq3} follows.
\end{proof}

We now focus on the integral term involving the Hamiltonian part of the equation.
\begin{lemma}
There exists $C_4 = C_4(d,\gamma,\delta,q)>0$ such that
\begin{equation} \label{mainineq4}
-\int_{\Omega_k} |\nabla u|^{\gamma-2} \, g(\nabla u,\nabla z) \, z_k^{\beta} \, \omega_{\mathrm{vol}} \leq
\tfrac13\beta \int_{\Omega_k} |\nabla z|^2z_k^{\beta-1} \, \omega_{\mathrm{vol}} 
+\tfrac14\Phi \int_{\Omega_k} z^{\eta}z_k^{\beta}\, \omega_{\mathrm{vol}}
+ C_4\int_{\Omega_k}z_k^{\beta+\eta}\, \omega_{\mathrm{vol}} \,\, .
\end{equation}
\end{lemma}

\begin{proof}
By the Young inequality
$$
-\int_{\Omega_k} |\nabla u|^{\gamma-2} \, g(\nabla u,\nabla z) \, z_k^{\beta} \, \omega_{\mathrm{vol}} \leq \tfrac{\beta}3 \int_{\Omega_k} |\nabla z_k|^2z_k^{\beta-1} \, \omega_{\mathrm{vol}} +\tfrac{3}{4\beta} \int_{\Omega_k} |\nabla u|^{2(\gamma-1)}z_k^{\beta+1} \, \omega_{\mathrm{vol}} \,\, .
$$
Since $\beta+\eta=\tfrac{p\gamma}{1+\delta}$ and $t^\frac12\leq h(t)^{\frac{1}{1+\delta}}$, using once more the weighted Young's inequality with the pair $(\eta,\frac{\eta}{\eta-1})$, we get 
$$\begin{aligned}
\tfrac{3}{4\beta} \int_{\Omega_k} |\nabla u|^{2(\gamma-1)} z_k^{\beta+1} \, \omega_{\mathrm{vol}} &\leq \tfrac{3}{4\beta} \int_{\Omega_k}z^{\frac{2(\gamma-1)}{1+\delta}} z_k^{\beta+1}\, \omega_{\mathrm{vol}} \\
&= \tfrac{3}{4\beta} \int_{\Omega_k} z^{\eta-1}z_k^{\frac{\beta(\eta-1)}{\eta}} z_k^{\frac{\beta}{\eta}+1}\, \omega_{\mathrm{vol}} \\
&\leq \tfrac14\Phi \int_{\Omega_k} z^{\eta}z_k^{\beta}\, \omega_{\mathrm{vol}} + C_4\int_{\Omega_k}z_k^{\beta+\eta}\, \omega_{\mathrm{vol}}
\end{aligned}$$
for some $C_4 = C_4(d,\gamma,\delta,q)>0$. Hence, we obtain \eqref{mainineq4}.
\end{proof}

We finally plug \eqref{mainineq1}, \eqref{mainineq2}, \eqref{mainineq3} \eqref{mainineq4} back to \eqref{mainineq}, so that 
\begin{multline} \label{mainineq5} 
\int_{\Omega_k} |\nabla z_k|^2\,z_k^{\beta-1}\, \omega_{\mathrm{vol}} \leq \,\,
 C_5\, \Bigg(\int_{\Omega_k} (1+|f|^p) \, \omega_{\mathrm{vol}} +\int_{\Omega_k}z_k^{\beta+\eta}\, \omega_{\mathrm{vol}} +\int_{\Omega_k} z_k^{\beta}\, \omega_{\mathrm{vol}} \\
+(\|f\|_{L^q(\Omega,\mathbb{R})})^p
\bigg(\int_{\Omega_k} (1+w)^{\frac{\delta}2\frac{pq}{q-p}} \, \omega_{\mathrm{vol}}\bigg)^{\frac{q-p}{q}}\Bigg)
\end{multline}
for some $C_5=C_5(d,\gamma,\delta,\kappa,q)>0$.

\begin{proof}[Proof of Proposition \ref{mainprop2}]

We start with the inequality \eqref{mainineq5}. We handle the term on the left-hand side of \eqref{mainineq5} via the Sobolev inequality \eqref{sob}, and infer the existence of a constant $C_6=C_6(d,\gamma,\delta,\sigma,q)>0$ such that
\begin{equation} \label{final2(1)}
\int_{\Omega_k} |\nabla z|^2z_k^{\beta-1} \, \omega_{\mathrm{vol}} =
\int_{\Omega_k} \big|\nabla \big(z_k{}^{\frac{\beta+1}{2}}\big) \big|^2 \, \omega_{\mathrm{vol}} \geq
C_6\, \bigg(\bigg(\int_{\Omega_k} z_k^{\frac{(\beta+1)d}{d-2}}\, \omega_{\mathrm{vol}}\bigg)^{\frac{d-2}d} -\int_{\Omega_k} z_k^{\beta+1}\, \omega_{\mathrm{vol}}\bigg) \,\, .
\end{equation}
The third term on the right-hand side of the equation can be estimated through the Young's inequality as
\begin{equation} \label{final2(2)}
\int_{\Omega_k} z_k^{\beta}\, \omega_{\mathrm{vol}} \leq \tfrac{\beta}{\beta+1}\int_{\Omega_k}z_k^{\beta+1}\, \omega_{\mathrm{vol}} +\tfrac1{\beta+1}\mathrm{vol}(\Omega_k) \,\, .
\end{equation}
We now discuss the remaining terms on the right-hand side of \eqref{mainineq5}. We choose $\delta>0$ sufficiently small so that $\delta \frac{p q}{q-p}<1$. Therefore, by the H\"older's inequality
\begin{equation} \label{final2(3)}
\int_{\Omega_k} (1+|f|^p) \, \omega_{\mathrm{vol}} \leq
{\rm vol}(\Omega_k)+ (\|f\|_{L^q(\Omega,\mathbb{R})})^p \, {\rm vol}(\Omega_k)^{\frac{q-p}{q}}
\end{equation}
and
\begin{equation} \label{final2(4)} \begin{aligned}
\bigg(\int_{\Omega_k} (1+w)^{\frac{\delta}2\frac{pq}{q-p}} \, \omega_{\mathrm{vol}}\bigg)^{\frac{q-p}{q}}
&\leq \bigg(\int_{\Omega_k} \sqrt{1+w} \, \omega_{\mathrm{vol}}\bigg)^{p\delta} {\rm vol}(\Omega_k)^{\frac{q-p}{q} -p\delta} \\
&\leq {\rm vol}(\Omega_k)^{\frac{q-p}{q}} +\left(\left\|\sqrt{1+\tfrac12|\nabla u|^2}\right\|_{L^1(\Omega,\mathbb{R})}\right)^{p\delta}{\rm vol}(\Omega_k)^{\frac{q-p}{q} -p\delta} \,\, .
\end{aligned} \end{equation}
Therefore, by \eqref{mainineq5} and \eqref{final2(1)}, \eqref{final2(2)}, \eqref{final2(3)}, \eqref{final2(4)} we get
\begin{equation} \label{mainineq6}
\bigg(\int_{\Omega_k} z_k^{\frac{(\beta+1)d}{d-2}}\, \omega_{\mathrm{vol}}\bigg)^{\frac{d-2}d} \leq
C_7\, \Bigg(
\int_{\Omega_k} z_k^{\beta+1}\, \omega_{\mathrm{vol}}
+\int_{\Omega_k}z_k^{\beta+\eta}\, \omega_{\mathrm{vol}}
+{\rm vol}(\Omega_k)
+{\rm vol}(\Omega_k)^{\frac{q-p}{q}}
+{\rm vol}(\Omega_k)^{\frac{q-p}{q} -p\delta}
\Bigg)
\end{equation}
for some $C_7=C_7(d,\gamma,\kappa,\rho,\sigma,q,K)$. Moreover, by suitably applying the weighted Young Inequality twice, we get, using that $\beta+\eta=\frac{p\gamma}{1+\delta}$ along with the relationes $p<q$ and $\frac{q\gamma}{1+\delta}=(\beta+1)\frac{d}{d-2}$, that
\begin{equation} \label{final2(4)}
C_7\int_{\Omega_k} z_k^{\beta+\eta}\, \omega_{\mathrm{vol}} \leq \tfrac12 \int_{\Omega_k} z_k^{\frac{(\beta+1)d}{d-2}}\, \omega_{\mathrm{vol}}
+C_8 {\rm vol}(\Omega_k)
\end{equation}
and
\begin{equation} \label{final2(5)}
C_7\int_{\Omega_k} z_k^{\beta+1}\, \omega_{\mathrm{vol}}\leq
\tfrac12 \int_{\Omega_k} z_k^{\frac{(\beta+1)d}{d-2}}\, \omega_{\mathrm{vol}}
+C_9 {\rm vol}(\Omega_k)
\end{equation}
for some positive constants $C_8,C_9$ depending only on $(d,\gamma,\kappa,\rho,\sigma,q)$. Therefore, by setting
$$
C \coloneqq 2\max{(C_7,C_8,C_9)} \quad \text{ and } \quad
\zeta(t) \coloneqq C\big(t +t^{\frac{q-p}{q}} +t^{\frac{q-p}{q} -p\delta}\big) \,\, ,
$$
the thesis follows from \eqref{mainineq6},\eqref{final2(4)} and \eqref{final2(5)}. We finally remark that, when $p$ defined in \eqref{p} does not satisfy the condition $p>2$, it is enough to take some $\tilde p$ satisfying $p<\tilde p<q$, with $\tilde p>2$ and argue as above. Then, the identity \eqref{bo2} becomes now the strict inequality
$$
(\beta+1)\frac{d}{d-2}>\frac{\gamma q}{1+\delta} \,\, .
$$
Therefore, one can conclude by a further application of H\"older's and Young's inequalities to achieve the same conclusion, which results in an additional term involving $\mathrm{vol}(\Omega_k)$ in $\zeta$ to prove Proposition \ref{mainprop2}.
\end{proof}

\subsection{Final remarks}

\begin{rem} 
Theorem \ref{main1} appears to be new even in the Euclidean setting $(M,g)=(\mathbb{R}^d,\langle\,,\rangle)$, with the condition
\begin{equation} \label{tildeD1} \tag{$\widetilde{D1}$}
\text{$\Omega \subset \mathbb{R}^d$ is a bounded, convex domain with boundary of class $\mathcal{C}^2$ and $d \geq 3$} \,\, ,
\end{equation}
in the presence of a linear first-order term with unbounded coefficient in Lebesgue spaces. A similar proof would lead to the same result (with the same restrictions on the integrability of the data) for problems posed on the flat torus (as analyzed in e.g. \cite{PV}), and the result seems new even in this framework.
\end{rem}

\begin{rem} \label{remEucl}
It is easy to realize that Theorem \ref{main2} yields maximal regularity for the Neumann problem \eqref{hjn} in the Euclidean space $(\mathbb{R}^d,\langle\,,\rangle)$ by assuming the hypothesis \eqref{tildeD1}, \eqref{In1}, \eqref{tildeIn2}. This geometric assumption on the domain is quite natural when dealing with gradient estimates for elliptic problems with Neumann boundary conditions, see \cite{Lio80,CirantJMPA,PorrCCM,PorLeo}, and at this stage we do not know whether it can be dropped. However, we mention that for problems with more regular data a different function $w$ allows to remove this geometric restriction \cite[Proposition 7.1]{PorLeo}, even for equations driven by mean curvature operators \cite[Lemma 2.3]{PorrCCM}.
\end{rem}

\begin{rem}\label{senzabdr}
Both Theorem \ref{main1} and Theorem \ref{main2} hold true when $(M,g)$ is a compact Riemannian manifold without boundary and $\Omega= M$. In this case, since $\partial \Omega = \emptyset$, the condition \eqref{D1} is trivially satisfied and one can avoid all the boundary terms coming from the integration by parts formula.
\end{rem}

\begin{rem}
We remark that, when $q\leq d\tfrac{\gamma-1}{\gamma}$, maximal regularity in general fails due to a counterexample found in \cite[Remark 1]{CGell}. We remark that the presence of a zero-th order term in the equation allows to get the maximal regularity property, at least when $\gamma<2$, at the borderline integrability threshold $\bar{q} \coloneqq d\frac{\gamma-1}{\gamma}$ (see \cite{CGpar}). We do not know whether a smallness condition for the norm of $f$ in $L^{\bar{q}}(\Omega,\mathbb{R})$ would allow to obtain the maximal regularity as well.
\end{rem}

\begin{rem}[Classical vs. strong solutions] \label{remClass-Str}
As remarked in e.g. \cite[Remark 3]{CGell} or \cite{BardiPerthame} in the Euclidean case, our Theorem \ref{main1} and Theorem \ref{main2} hold for strong solutions of class $W^{2,q}(\Omega,\mathbb{R}) \cap W^{1,q\gamma}(\Omega,\mathbb{R})$ and avoid to require classical regularity on the data (so, for instance, assumptions \eqref{In1} and \eqref{In2} can be dropped) arguing from a variational viewpoint, as done in \cite{CGell}.
\end{rem}

\begin{rem}[General first-order terms] \label{Hgen}
As discussed in \cite[Remark 4]{CGell}, one can treat more general Hamiltonians in Theorem \ref{main1} and Theorem \ref{main2} . For example, both results apply to equations of the form
$$
-\Delta u(x)+ H(x,{\rm d} u(x))=f(x)
$$
where $H : T^*M \to \mathbb{R}$ satisfies the following property: there exists an exponent $\gamma>1$ and two positive constants $c_1,c_2 >0$ such that
$$
|H(x,p)-c_1|p|^\gamma|\leq c_2 \quad \text{ for all } (x,p) \in T^*M \,\, .
$$
For instance, the Hamiltonian
$$
H(x,p) \coloneqq c_1|p|^\gamma -h(x) +\lambda \,\, ,
$$
with $c_1>0$, $h : \Omega \to \mathbb{R}$ bounded and $\lambda \in \mathbb{R}$ satisfying $\lambda \geq \lambda_0$ for some fixed constant $\lambda_0 \in \mathbb{R}$, fits within the setting of Theorems \ref{main1} and \ref{main2}. Indeed, it is sufficient to consider the equation
$$
-\Delta u +c_1|\nabla u|^\gamma = f_H \,\, , \quad \text{ with } f_H(x) \coloneqq f(x) -\lambda -h(x)
$$
and observe thatm under the standing assumptions, it holds $\|f_H\|_{L^q(\Omega,\mathbb{R})} \leq \|f\|_{L^q(\Omega,\mathbb{R})} +c$ for some constant $c>0$ independent of $u$ and depending only on $\lambda_0$ and $\inf_\Omega h$. Nonetheless, our Theorem \ref{main1} allows even the presence of unbounded ingredients. In particular, one might consider more general Hamiltonian terms depending explicitly on $x$ and satisfying
$$
|\partial_{xp}^2H(x,p)|\leq C(1+|p|^{\gamma-1}) \quad \text{ for all $(x,p) \in T^*M|_{\overline{\Omega}}$} \,\, ,
$$
where $C>0$ is a constant and $\partial_{xp}^2$ denotes the mixed partial derivatives with respect to the first and second entry, respectively. This results in the presence of an additional integral term that can be treated as in Lemma 4.3 in \cite{PV}.
\end{rem}

\begin{rem}
An integral approach similar to Theorem \ref{main2} can be found in \cite{CM} for equations driven by the $p$-Laplacian without lower-order perturbations. Indeed, the technique in \cite{CM} relies on multiplying the equation against $\Delta u$ and integrating over the super-level sets of the gradient. This heuristically corresponds to the method used in Theorems \ref{main1} and \ref{main2} with the choice $\beta=0$ and $h'(t) \equiv 1$. In particular, if one puts the argument e.g. of Theorem \ref{main1} in a variational setting, the underlying idea is to use as a test function a $p$-Laplacian with a suitable (large) $p$. This refinement is needed due to the presence of a first-order term with superlinear growth, unlike \cite{CM}.
\end{rem}

\medskip
\section{Applications to stationary Mean Field Games systems} \label{sect_4}
\setcounter{equation} 0

In this section, we restrict to the Euclidean setting $(M,g)=(\mathbb{R}^d, \langle \, , \rangle)$ and we apply our Theorem \ref{main2} (see Remark \ref{remEucl}) to the following Mean Field Games (MFG) equipped with Neumann boundary conditions on convex domains:
\begin{equation} \label{mfg}
\begin{cases}
-\Delta u(x) +H(x,{\rm d}u(x))+\lambda=V(m(x)) & x \in \Omega \,\, , \\
-\Delta m(x) -\mathrm{div}\big(\partial_pH(x,{\rm d}u(x))m(x)\big)=0 & x \in \Omega \,\, , \\
\partial_\nu u (x) =0 \,\, , \quad \partial_\nu m(x) +\big(\partial_pH(x,{\rm d}u(x)) \cdot \nu\big)\,m(x) =0 & x \in \partial \Omega \,\, , \\
\int_\Omega m(x)\,{\rm d}x=1\ , \quad m(x)>0 & x \in \overline{\Omega} \,\, .
\end{cases}
\end{equation}
Here, the unknowns are $(u,\lambda,m)$, $\lambda$ being the so-called {\it ergodic constant}, and $\partial_{p}$ denotes the partial derivatives of $H(x,p)$ with respect to the second entry. This coupled system of PDEs arises within the theory of Mean Field Games introduced in the mathematical community by J.-M. Lasry and P.-L. Lions \cite{LL07}, which models differential games with infinitely many indistinguishable rational agents. In particular, system \eqref{mfg} describes the so-called Nash equilibria of a population of agents which aims at minimizing some long-time average criterion.

In this section we drop the periodicity condition, which is usually considered in most of the literature to avoid technicalities, and treat problems confined in a domain $\Omega$ satisfying the condition \eqref{tildeD1}. In particular, this means that the state space for the players is set on $\Omega$ and their trajectories follow a stochastic differential equation with brownian noise $\sqrt{2}B_t$ with reflection at the boundary $\partial\Omega$. We refer to \cite{CirantJMPA,Alpar} for a detailed discussion on such MFG systems, and for further references on the subject.

Our main result for \eqref{mfg} states the existence of classical solutions when the function $V$ depends locally on the density $m$, i.e. it is a local coupling function, in the so-called {\it defocusing regime} \cite{CCPDE}. In particular, we will assume that
\begin{equation} \label{MFG1} \tag{MFG1}
\begin{gathered}
\text{$V: [0,+\infty) \to \mathbb{R}$ is of class $\mathcal{C}^1$ and there exist \,$\alpha>0$\,, \,$C_V>1$\, such that } \\
\text{$C_V^{-1} m^{\alpha-1} \leq V'(m) \leq C_V (m^{\alpha-1}+1)$\, for any \,$m \geq 0$\, .}
\end{gathered} \end{equation}
We also assume that
\begin{equation} \label{MFG2} \tag{MFG2}
\begin{gathered}
\text{$H(x,{\rm d}u(x))=\tfrac1{\gamma} |\nabla u(x)|^\gamma-b(x)$\, for some \,$b \in \mathcal{C}^2(\overline{\Omega},\mathbb{R})$ with \,$\partial_\nu b \geq 0$\, on \,$\partial\Omega$}
\end{gathered} \end{equation}
and 
\begin{equation} \label{MFG3} \tag{MFG3}
\gamma>\tfrac{d}{d-2} \,\, , \quad \alpha<\begin{cases}
\infty&\text{ if }d=3\\
\tfrac{\gamma'}{d-2-\gamma'}&\text{ if }d\geq 4
\end{cases} \,\, ,
\end{equation}
where $\gamma' = \tfrac{\gamma}{\gamma-1}$. Here, we restricted, for simplicity, to $\gamma>\tfrac{d}{d-2}$ to have the sole restriction$ q>\frac{d}{\gamma'}$ from Theorem \ref{main2}. One can weaken the condition on $\gamma$ lowering the constraint on $\alpha$, but we avoid this technical step.

\begin{theorem} \label{mainappl}
If hypotheses \eqref{tildeD1}, \eqref{MFG1}, \eqref{MFG2}, \eqref{MFG3} hold, then there exists a classical solution $(u,\lambda,m)$ to the Neumann problem \eqref{mfg}.
\end{theorem}

\begin{proof}
We claim that, under the above restrictions on $\alpha$ one can prove a priori bounds for second-order derivatives of solutions of \eqref{mfg}. Indeed, existence theorems can be obtained by implementing a regularization procedure along with a fixed-point theorem that exploits the duality structure of the system, see e.g. \cite[Theorems 4 and 7]{CirantJMPA}. The former consists in replacing $V$ with a regularizing functional
$$
V_\epsilon(x) \coloneqq V(m\star \chi_\epsilon(x)) \star \chi_\epsilon(x) \,\, , \quad x \in \overline{\Omega}
$$
where $\chi$ is a positive, symmetric mollifier with ${\rm supp}(\chi) \subset \Omega$ and $\chi_{\epsilon}(x) \coloneqq \epsilon^{-d}\chi(\epsilon^{-1}x)$. Here, we denoted by $\star$ the convolution of functions. One then proves bounds on the sequence $(u_\epsilon,\lambda_\epsilon,m_\epsilon)$ that are independent of $\epsilon$ and pass to the limit as $\epsilon \to 0^+$. Note that $m$ in the second equation satisfies $\int_\Omega m=1$ and $m>0$. We now detail how to get the conclusion via standard elliptic regularity and bootstrapping arguments once bounds for
$$
|D^2u_\epsilon| \, , \,\, |\nabla u_\epsilon|^\gamma \in L^q(\Omega,\mathbb{R}) \quad \text{ for some } \,\, q>\tfrac{d(\gamma-1)}{\gamma}
$$
are established, assuming the a priori information $V(m)\in L^q(\Omega,\mathbb{R})$. First, note that
$$
|\nabla u_\epsilon|^\gamma \in L^q(\Omega,\mathbb{R}) \, , \,\, q>\tfrac{d(\gamma-1)}{\gamma} \quad \Longrightarrow \quad
\partial_pH({\rm d}u) \in L^s(\Omega,\mathbb{R}^d) \,\, \text{ for some } s>d \,\, .
$$
Then, one can invoke \cite[Proposition 12]{CirantJMPA}, applied with $q=\frac{d}{d-1}$, to prove that $m \in L^{\infty}(\Omega,\mathbb{R})$. We can then plug this information back to the Hamilton-Jacobi equation and conclude that $V(m) \in L^{\infty}(\Omega,\mathbb{R})$, and finally use \cite[Theorem 19]{CirantJMPA} or Theorem \ref{main1} to get $u \in W^{2,r}(\Omega,\mathbb{R})$ for $r>d$. This implies, by Sobolev embeddings, that $u\in \mathcal{C}^{1,\delta}(\Omega,\mathbb{R})$ for some $\delta\in(0,1)$, so that $\nabla u$ is H\"older continuous. This allows then to regard the Hamilton-Jacobi equation as a Poisson equation with source term in a H\"older space. We then apply the Schauder estimates to conclude the claimed regularity for $u$. Moreover, the drift of the Fokker-Planck equation is now H\"older continuous, and hence one gets $m \in \mathcal{C}^{1,\sigma}(\Omega,\mathbb{R})$ for some $\sigma\in(0,1)$. The equations are now linearized and one can go beyond and get additional regularity by bootstrapping arguments (see e.g. \cite[Step 5 in Theorem 7]{CirantJMPA}).

We now briefly describe how to set up the procedure and get bounds of $V(m)\in L^q(\Omega,\mathbb{R})$ independently of $u,m$ and $\epsilon$. For the sake of notation, in the following we will denote by ${\rm d}x$ the $d$-dimensional Hausdorff measure on $\mathbb{R}^d$ and by ${\rm d}s$ the $(d{-}1)$-dimensional Hausdorff measure on $\partial \Omega$. First, arguing as in \cite[Lemma 11]{CirantJMPA} we claim that there exists a constant $C>0$ such that
\begin{equation}\label{estsob}
\int_{\Omega}|m\star \chi_\epsilon|^{\frac{d(\alpha+1)}{d-2}}\,{\rm d}x \leq C \quad \text{ for any } \epsilon>0 \,\, ,
\end{equation}
i.e. $\|m\|_{L^{\frac{d(\alpha+1)}{d-2}}(\Omega,\mathbb{R})} < \infty$. This implies that
$$
\|m^\alpha\|_{L^{(1+\frac1\alpha)\frac{d}{d-2}}(\Omega,\mathbb{R})}<\infty \,\, ,
$$
so this ensures the bound of $V(m)$ in the same Lebesgue space due to the standing hypotheses on the coupling. Moreover, $\lambda$ is bounded by the maximum principle, see e.g. \cite[Theorem 7]{CirantJMPA}. Therefore, we can apply Theorem \ref{main2}, combined with Remark \ref{Hgen} (note that $L^1$-bounds on $\nabla u$ are straightforward by integrating the equation), and conclude that $|\nabla u|^\gamma\in L^{(1+\frac1\alpha)\frac{d}{d-2}}(\Omega,\mathbb{R})$ provided that \eqref{MFG3} is satisfied (note that, when $d=3$, the condition $(1+\frac1\alpha)\frac{d}{d-2}>\frac{d}{\gamma'}$ is always realized for any $\alpha>0$). Then, $\partial_pH({\rm d}u)\in L^r(\Omega,\mathbb{R}^d)$, for some $r>d$ and one can proceed as outlined at the beginning of the proof to get the conclusion.

We now discuss \eqref{estsob}. We test the Hamilton-Jacobi equation with $\Delta m$, the Fokker-Planck equation with $\Delta u$ and use the boundary conditions to find 
\begin{multline} \label{mainid}
-\int_\Omega \mathrm{Tr}\big(\partial_{pp}^2H(\nabla u)\, (D^2u)^2\big)\,m\,{\rm d}x
+\tfrac1\gamma \int_{\partial\Omega}m\,\partial_\nu |\nabla u|^\gamma\,{\rm d}s
= \\
= \int_\Omega \nabla [V(m\star\chi_\epsilon)\star \chi_\epsilon]\cdot \nabla m \,{\rm d}x+\int_\Omega \nabla b\cdot \nabla m\,{\rm d}x \,\, .
\end{multline}
Since $\Omega$ is convex, we have $\partial_\nu |\nabla u|^\gamma\leq0$ on $\partial\Omega$, so that $\int_{\partial\Omega}m\,\partial_\nu |\nabla u|^\gamma\,{\rm d}s\leq0$ since $m>0$, while the properties of the convolutions readily imply
$$
\int_\Omega \nabla [V(m\star\chi_\epsilon)\star \chi_\epsilon]\cdot \nabla m\,{\rm d}x
=\int_\Omega \nabla [V(m\star\chi_\epsilon)]\cdot \nabla (m\star \chi_\epsilon)\,{\rm d}x
$$
Then, we integrate by parts the last term in \eqref{mainid}, use the conditions \eqref{MFG2}, $\int_\Omega m=1$ and $m>0$ to find
$$
\int_\Omega \nabla [V(m\star\chi_\epsilon)]\cdot \nabla (m\star \chi_\epsilon)\,{\rm d}x
\leq -\int_{\partial\Omega} m\,\partial_\nu b\,{\rm d}s +\int_\Omega m\,\Delta b\,{\rm d}x
\leq \| D^2b \|_{L^\infty(\Omega,{\rm Sym}_d(\mathbb{R}))} \,\, .
$$
We then apply the chain rule to $V$ and find 
$$
\int_\Omega V'(m\star \chi_\epsilon)\nabla(m\star \chi_\epsilon)\cdot \nabla(m\star \chi_\epsilon)\,{\rm d}x
\leq \| D^2b \|_{L^\infty(\Omega,{\rm Sym}_d(\mathbb{R}))} \,\, .
$$
Finally, \eqref{MFG1} leads to
$$
\int_\Omega (m\star \chi_\epsilon)^{\alpha-1}|\nabla(m\star \chi_\epsilon)|^2 \,{\rm d}x
= \int_\Omega |\nabla[(m\star\chi_\epsilon)^{\frac{\alpha+1}{2}}]|^2\,{\rm d}x
\leq C_1 \| D^2b \|_{L^\infty(\Omega,{\rm Sym}_d(\mathbb{R}))}
$$
and then by the Sobolev inequality
$$
\int_\Omega m^{(\alpha+1)\frac{d}{d-2}}\,{\rm d}x \leq C_2 \| D^2b \|_{L^\infty(\Omega,{\rm Sym}_d(\mathbb{R}))} \,\, ,
$$
which concludes the proof.
\end{proof}

\begin{rem}[Comparison with the literature]\label{rem:exp}
The results in \cite{CirantJMPA} state that one has classical regularity for \eqref{mfg}, cf \cite[Remark 8]{CirantJMPA}, when
$$
\alpha<\tfrac{1}{d-3} \quad \text{ with $d>3$ } \,\, .
$$
This range was later improved by the same author in \cite[Theorem 1.4]{CCPDE}, where it was proved the smoothness of solutions to \eqref{mfg} when 
$$
\alpha<\tfrac{\gamma'}{d-\gamma'} \,\, , \quad \gamma>\tfrac{d}{d-1} \,\, .
$$
It is immediate to see that for any $\gamma>1$ we have the inequality
$$
\max \big\{\tfrac{1}{d-3},\tfrac{\gamma'}{d-\gamma'}\big\} < \tfrac{\gamma'}{d-2-\gamma'} \,\, ,
$$
and therefore our results provide a wider range for the exponent $\alpha$ guaranteeing the smoothness of solutions to these PDE systems. We mention that \cite[Proposition 3.6]{Alpar} provides (local) classical regularity for the Neumann problem when $\gamma<\frac{d}{d-1}$ via different methods of variational nature that allows to treat only slowly increasing first-order terms, while classical regularity of solutions has been obtained in \cite[Theorem 1.4]{CCPDE} in the same range for $\gamma$ for any $\alpha<\infty$ via different PDE arguments. We emphasize that the approach used here is designed to handle first-order terms growing more that $\frac{d}{d-1}$ (indeed $\frac{d}{\gamma'}>1$ whenever $\gamma>\frac{d}{d-1}$), so our results complement \cite[Theorem 1.4]{CCPDE} and \cite[Proposition 3.6]{Alpar}.
\end{rem}

\begin{rem}[Multi-population MFGs]
The results can be extended to more general MFG structures, such as for multi-population MFG systems, as studied in \cite{CirantJMPA,Alpar}, under the same assumptions of \cite[Theorem 7]{CirantJMPA} for the coupling $V(m)$, improving the existent results in the literature. We do not detail the proof to gain classical regularity being the same of the case of one population described above.
\end{rem}

\begin{rem}
The same method through second order estimates allow to treat second order Mean Field Games systems posed on compact Riemannian manifolds without boundary or, more generally, PDE systems on the same framework of Theorem \ref{main2}. In these directions, some results for second order stationary Mean Field Games on sub-Riemannian geometries with more regular data can be found in \cite{DF}.
\end{rem}

\medskip


\end{document}